\title{Icosahedral invariants and Shimura curves}
\author{Atsuhira Nagano}

\documentclass[10pt]{article}
\usepackage{mathrsfs,bm,graphics,amsfonts,amsthm,amssymb,amsmath,amscd,booktabs}
\usepackage[dvips]{graphicx,color}
\def\bigzerou{\smash{\lower1.7ex\hbox{\b 0}}}

\newtheorem{thm}{Theorem}[section]
\newtheorem{df}{Definition}[section]
\newtheorem{lem}{Lemma}[section]
\newtheorem{prop}{Proposition}[section]
\newtheorem{rem}{Remark}[section]
\newtheorem{exap}{Example}[section]

\def\comment#1{{ }}
\makeatletter
  
      \@addtoreset{equation}{section}
        
      \makeatletter

\begin{document}

\maketitle

\begin{abstract}

Shimura curves are  moduli spaces of abelian surfaces with quaternion multiplication.
Models of Shimura curves are very important in number theory.
Klein's icosahedral invariants $\mathfrak{A},\mathfrak{B}$ and $\mathfrak{C}$ give the Hilbert modular forms for $\sqrt{5}$ via the period mapping for a family of $K3$ surfaces.
Using the period mappings for several families of $K3$ surfaces, we obtain explicit models of Shimura curves with small discriminant in the weighted projective space ${\rm Proj} (\mathbb{C}[\mathfrak{A},\mathfrak{B},\mathfrak{C}])$.
 
\end{abstract}

\footnote[0]{Running head: Icosahedral invariants and Shimura curves}
\setlength{\baselineskip}{14 pt}

\section*{Introduction}
This paper gives an application of the moduli theory of $K3$ surfaces to the number theory.
We obtain an explicit relation among   abelian surfaces with quaternion multiplication,   Hilbert modular functions and periods of $K3$ surfaces.

The moduli spaces for principally polarized abelian surfaces  determined by the structure of  the ring of endomorphisms are very important in number theory (\cite{Geer} Chapter IX Proposition (1.2), see also Table 1).
In this paper, we study the moduli spaces of principally polarized abelian surfaces with quaternion multiplication (For the detailed definition, see Section 2.2).
They are called Shimura curves.
\begin{table}[h]
\center
\begin{tabular}{cccc}
\toprule
Abelian surface $A$ & ${\rm End}_0 (A)$ & Moduli Space   \\
 \midrule
 Generic  & Rational field & Igusa $3$-fold $\mathcal{A}_2$ \\ 
Real multiplication & Real quadratic field & Humbert surface $\mathcal{H}$    \\
Quaternion multiplication & Quaternion algebra& Shimura curve $\mathcal{S}$  \\
Complex multiplication & CM field &  CM points  \\
 \bottomrule
\end{tabular}
\caption{ The moduli spaces of abelian surfaces }
\end{table}

To the best of the author's knowledge, to obtain explicit models of Shimura curves  is a non trivial problem because Shimura curves have no cusps.
In this paper, we shall obtain new models of  Shimura curves for quaternion algebras with small discriminant.
We  consider the weighted projective space $\mathbb{P}(1:3:5)={\rm Proj} (\mathbb{C}[\mathfrak{A},\mathfrak{B},\mathfrak{C}]),$ where $\mathfrak{A},\mathfrak{B} $ and $\mathfrak{C}$ are Klein's icosahedral invariants of weight $1,3$ and $5$ respectively.
We shall give the explicit defining equations of the Shimura curves for small discriminant  in $ {\rm Proj}(\mathbb{C}[\mathfrak{A},\mathfrak{B}, \mathfrak{C}]) $.

Here, let us see  the reason why we consider the icosahedral invariants.
The moduli space $\mathcal{H}_\Delta $ of principally polarized abelian surfaces with real multiplication by $\mathcal{O}_\Delta$ is called the Humbert surface (for detail, see Section 1.1).
The Humbert surface $\mathcal{H}_\Delta$ is uniformized by Hilbert modular functions for $\Delta$.
Among Humbert surfaces, the case for $\mathbb{Q}(\sqrt{5})$ is the simplest,  since its discriminant is the smallest.
In \cite{NaganoTheta},  we studied the family  $\mathcal{F}=\{ S(\mathfrak{A}:\mathfrak{B}:\mathfrak{C})\}$  of elliptic $K3$ surfaces. 
We can regard $\mathcal{F}$ as a family parametrized over $\mathcal{H}_5$.

By the way, the Igusa 3-fold $\mathcal{A}_2$ is the moduli space of principally polarized abelian surfaces.
The family  $\mathcal{F}_{CD}=\{S_{CD}(\alpha:\beta:\gamma:\delta)\}$ of $K3 $ surfaces for $(\alpha:\beta:\gamma:\delta)\in \mathbb{P}(2:3:5:6)$  is studied by  Kumar \cite{Kumar}, Clingher and Doran \cite{CD} and \cite{NaganoShiga}.
 This family $\mathcal{F}_{CD}$ is parametrized over $\mathcal{A}_2.$ 
Our family $\mathcal{F}$ can be regarded as  a subfamily of $\mathcal{F}_{CD}$. 
However,  it is not apparent to  describe the embedding $\mathcal{F} \hookrightarrow \mathcal{F}_{CD}$ explicitly.
Our first result of this paper is to obtain the embedding $\Psi_5 : \mathbb{P}(1:3:5) \hookrightarrow \mathbb{P}(2:3:5:6)$ of the parameter spaces (see Theorem \ref{PsiThm}).

A Shimura curve $\mathcal{S}$ is a $1$-dimensional subvariety of  $\mathcal{A}_2.$
In several cases,  $\mathcal{S}$ is contained in the image  $\Psi_5(\mathbb{P}(1:3:5))$ and the pull-back $\Psi_5^*(\mathcal{S})$  is a curve in $\mathbb{P}(1:3:5)$.
In this paper, we obtain models of $\Psi_5^*(\mathcal{S})$ for such cases.
We note that $\mathcal{S}$ and $\Psi_5^*(\mathcal{S})$ are isomorphic as varieties.
Good modular properties of  ${\rm Proj}(\mathbb{C}[\mathfrak{A},\mathfrak{B}, \mathfrak{C}])$ enable us to study $\Psi_5^*(\mathcal{S})$ effectively.
Also, our study is  based on the results  of quaternion algebras due to Hashimoto \cite{Hashimoto} 
and  elliptic $K3$ surfaces due to  Elkies and Kumar \cite{EK}.

As a result, we  obtain the following explicit defining equations for the Shimura curves $\Psi_5^*(\mathcal{S})$ for  discriminant $6,10,14$ and $15$. 
Setting $\displaystyle X=\frac{\mathfrak{B}}{\mathfrak{A}^3},Y=\frac{\mathfrak{C}}{\mathfrak{A}^5},$ $(X,Y)$ are affine coordinates of $\mathcal{H}_5$.
We have the models (see Theorem \ref{ShimuraModular}, \ref{ThmS6S10},  \ref{Thm5,12} and \ref{ThmS14S15}):
\begin{align*}
\begin{cases}
&\Psi_5^*(\mathcal{S}_6): 3125  - 3375 X + 243 Y =0,\\
&\Psi_5^*(\mathcal{S}_{10}): 1 - 5  X + Y =0,\\
&\Psi_5^*(\mathcal{S}_{14}):30517578125  + 911865234375  X + 42529296875  X^2 -  97897974609375  X^3 \notag \\
  & \quad \quad\quad \quad + 424490000000000  X^4 - 
  345600000000000 X^5 + 2383486328125  Y \notag \\
&\quad \quad\quad \quad + 
  32875975781250  X Y - 147816767984375  X^2 Y + 
  228155760000000  X^3 Y \notag \\
  &\quad \quad\quad \quad + 19189204671875  Y^2 - 
  29675018141125  X Y^2 + 344730881243 Y^3=0,
\\
&\Psi_5^*(\mathcal{S}_{15}):30517578125  + 911865234375  X + 42529296875  X^2 - 
  97897974609375  X^3 \notag \\
  & \quad \quad \quad \quad+ 424490000000000 X^4 - 
  345600000000000 X^5 + 2383486328125 Y=0.
\end{cases}
\end{align*}

Some researchers obtained models of Shimura curves (for example, Kurihara \cite{Kurihara}, Hashimoto and Murabayashi \cite{HM}, 
Besser \cite{Besser}, Elkies \cite{Elkies1}, \cite{Elkies2}, Kohel and Verrill \cite{KohelVerrill}, Voight \cite{Voight}, Bonfanti and van Geemen \cite{BG}).
In comparison with already known models, our new models have the following features.

\begin{itemize}

\item
They are closely related to the classical invariant theory.
Namely, our coordinates $\mathfrak{A},\mathfrak{B},\mathfrak{C}$  of the common ambient space of Shimura curves are coming from Klein's icosahedral invariants.
Especially,  the Shimura curves for  discriminant $6$ and $10$ have very simple forms.
These two curves are just lines touching the locus of Klein's icosahedral equation (see Figure 5).

\item
The moduli  of our family $\mathcal{F}$ of $K3$ surfaces were studied in detail.
We have an explicit  expression of the period mapping (\cite{NaganoTheta},\cite{NaganoKummer})  and the Gauss-Manin connection (\cite{NaganoPDE}) for $\mathcal{F}$.
 These properties are very useful to study  Shimura curves effectively (for example, see the proof of Theorem   \ref{ThmS6S10}).

\item
Shimura (\cite{S97}) studied unramified class fields over CM fields of certain types.
In \cite{NaganoCM}, an explicit construction of such class fields over quartic CM fields using the special values of $X$ and $Y$ is given.

\item
In fact, our $K3$ surface $S(\mathfrak{A}:\mathfrak{B}:\mathfrak{C})$ is a toric hypersurface.
To study the mirror symmetry for toric $K3$ hypersurfaces is an interesting problem in recent geometry and physics.
In \cite{HNU}, our $S(\mathfrak{A}:\mathfrak{B}:\mathfrak{C})$ are studied from the viewpoint of mirror symmetry.
Especially, our parameters $X $ and $Y$ are directly related  to the secondary stack for the  toric $K3$ hypersurface.

\end{itemize}
Thus,  our new models of Shimura curves  are naturally related to various topics.

\section{Moduli of principally polarized abelian surfaces}

\subsection{ Principally polarized abelian surfaces with real multiplication}

Let $\mathfrak{S}_2$ be the Siegel upper half plane of rank $2$.
Let us consider a principally polarized abelian variety $(A,\Theta)$ with the theta divisor $\Theta$ and the period matrix $(\Omega, I_2)$, where $\Omega \in \mathfrak{S}_2$.
The symplectic group $Sp(4,\mathbb{Z})$ acts on $\mathfrak{S}_2.$
The quotient space $Sp(4,\mathbb{Z}) \backslash \mathfrak{S}_2$ gives the moduli space $\mathcal{A}_2$ of principally polarized abelian varieties.
This is called the Igusa $3$-fold.

The ring of endomorphisms is given by
$
{\rm End}(A)=\{a \in M_2(\mathbb{C}) |   a (\Omega,I_2)=(\Omega,I_2)M  \text{ for some } M\in M(4,\mathbb{Z}) \}.
$
The principal polarization given by $\Theta$ induces the alternating Riemann form $E(z,w)$.
Set ${\rm End}_0 (A)={\rm End}(A) \otimes _\mathbb{Z} \mathbb{Q}$.
If $A $ is a simple abelian variety, then ${\rm End}_0(A)$ is a division algebra. 
The Rosati involution $a \mapsto a^\circ $ is an involution on  ${\rm End}_0(A)$   
and gives an adjoint of the alternating Riemann form: 
$
E(a z ,w)=E(z,a^\circ w).
$
Note that the Rosati involution satisfies ${\rm Tr}(a a^\circ)>0$.

 A point $\Omega=\begin{pmatrix} \tau_1 & \tau_2 \\ \tau_2  & \tau_3 \end{pmatrix}\in \mathfrak{S}_2$ is said to have the singular relation with the invariant $\Delta$ 
if there exist relatively prime integers $a,b,c,d$ and $e$ such that the following equations hold:
\begin{align}\label{singularR}
 a \tau_1 + b \tau_2 + c \tau_3 + d (\tau_2^2 -\tau_1 \tau_3) +e =0, \quad
 \Delta = b^2 - 4 a c - 4 d e
\end{align}

\begin{df}\label{HumbertDf}
Set $N_{\Delta} =\{  \tau \in \mathfrak{S}_2 |  \tau \text{ has a  singular relation with } \Delta \}.$
The image of $N_\Delta$ under the canonical mapping $ \mathfrak{S}_2 \rightarrow  Sp(4,\mathbb{Z})\backslash \mathfrak{S}_2$ is called the Humbert surface of invariant $\Delta$
\end{df}

Let $\mathcal{O}_\Delta$ be the ring of integers of the field $\mathbb{Q}(\sqrt{\Delta})$. 
The Humbert surface of invariant $\Delta$ gives the moduli space of principally polarized abelian surfaces $(A,\Theta)$ with $\mathcal{O}_\Delta \subset {\rm End} (A)$ and $\mathbb{Q}(\sqrt{\Delta})\cap {\rm End}(A)=\mathcal{O}_\Delta$. 
Such an abelian surface is said to have  real multiplication by $\mathcal{O}_\Delta$ (see \cite{Geer} or \cite{Hashimoto}).

\subsection{Quaternion multiplication and Shimura curves}

In this subsection, we recall the properties of Shimura curves. For detail, see    \cite{Rotger}, \cite{Rotger2} or   \cite{Vigneras}.

Let $B$  be an indefinite quaternion algebra over $\mathbb{Q}$ with $B\not \simeq M_2(\mathbb{Q})$.
We have an isomorphism $B\otimes _\mathbb{Q} \mathbb{R} \simeq M_2(\mathbb{R})$.
Let $p_1,\cdots ,p_t$ be the distinct primes at which $B$ ramifies.
We can show  that $t\in 2\mathbb{Z}$.
The number 
$D=p_1 \cdots p_t$ is called the discriminant of $B$.
Two quaternion algebras $B$ and $B'$ are isomorphic as $\mathbb{Q}$-algebras 
if and only if the discriminant of $B$ coincides with that of $B'$.

For $\alpha\in B$,
let $\alpha \mapsto \alpha'$ be the canonical involution  defined by $\alpha'={\rm Tr}_{B/\mathbb{Q}}(\alpha)-\alpha$.
An element $\alpha\in B $ is called integral if both ${\rm Tr_{B/\mathbb{Q}} (\alpha)}$ and ${\rm Nr_{B/\mathbb{Q}} (\alpha)}$ are in $\mathbb{Z}$.
If a subring $\mathfrak{O} (\subset B)$ of integral elements is a finitely generated  $\mathbb{Z}$-module  of $B$  and satisfies $\mathbb{Q} \mathfrak{O}=B$,
we call $\mathfrak{O}$ an order of $B$. 
A maximal order is an order that is maximal under inclusion.
We note that a maximal order in $B$  is unique up to conjugation.

For a maximal order $\mathfrak{O}$ in $B$ with  discriminant $D$,
we put
$
\Gamma^{(1)} =\{\gamma\in \mathfrak{O}| {\rm Nr}_{B/\mathbb{Q}}(\gamma)=1\}.
$
The group $\Gamma^{(1)}$ gives a discrete subgroup of $SL(2,\mathbb{R})$.
For  $\rho\in \mathfrak{O}$ satisfying $\rho^2<0$, 
we have an involution  given by
$
i_\rho: \alpha \mapsto \rho^{-1} \alpha' \rho.
$
For $\xi,\eta \in B$, put $E_\rho(\xi,\eta)={\rm Tr}(\rho \xi \eta').$
The pairing $E_\rho$ gives a skew symmetric form on $B$.
Moreover, we can show that
for $\gamma \in \Gamma^{(1)}$ we have
\begin{align}\label{norm1E}
E_\rho(\xi \gamma,\eta\gamma)= E_\rho(\xi,\eta).
\end{align}

\begin{df}\label{QMDef}
Take $\rho\in \mathfrak{O}$ such that $\rho^2 =-D  $ and $\rho \mathfrak{O}=\mathfrak{O} \rho$.
We call $(A,\Theta,\iota)$ 
a principally polarized abelian surface
with quaternion multiplication by
 $(\mathfrak{O},
 i_\rho)$ 
if
 $\iota:\mathfrak{O}\hookrightarrow {\rm End}(A) $ and 
 $\iota(\alpha)\mapsto \iota \circ i_\rho(\alpha)$ coincides with 
 the Rosati involution $\iota(\alpha) \mapsto \iota(\alpha)^\circ$.
\end{df}

The quotient space $\Gamma^{(1)} \backslash \mathbb{H}$  for $B$ of discriminant $D$  
is already a compact Riemann surface.
This is called the Shimura curve $\mathbb{S}_D$ for $B_D$.
Shimura 
proved that $\mathbb{S}_D$ is isomorphic to the moduli space of $(A,\Theta,\iota)$ with quaternion multiplication by $(\mathfrak{O},i_\rho)$.
Such Shimura curves were firstly studied in \cite{S67}, \cite{S70} and \cite{S75}.
There exists a quaternion modular embedding (see the following diagram):
\begin{align}\label{Phi1}
\Omega
:\quad \mathbb{H} \rightarrow \mathfrak{S}_2; \quad w \mapsto \Omega(w).
\end{align}

\vspace{-5mm}
\begin{align*}
\begin{CD}
\mathbb{H} @> \Omega >>\mathfrak{S}_2\\
@V  \Gamma^{(1)} VV @VV  Sp(4,\mathbb{Z})V\\
\mathbb{H} @>> \Omega > \mathfrak{S}_2   
 \end{CD}
\end{align*}

We note that the Shimura curve $\mathbb{S}_D$ does not be embedded in the Igusa $3$-fold $\mathcal{A}_2$.
There exist a projection $\varphi_D:\mathbb{S}_D\rightarrow \mathcal{S}_D \subset \mathcal{A}_2$ and
a group $W_D \subset {\rm Aut}(\mathbb{S}_D)$ such that  $\mathcal{S}_D$ is birationally equivalent to $\mathbb{S}_D/W_D$.
We note that
$\varphi_D$ is  either generically $2$ to $1$ or generically  $4$ to $1$. 
In this paper, we also call the curve $\mathcal{S}_D$ the Shimura curve.
The curve $\mathcal{S}_D$ gives isomorphism classes of $(A,\Theta)$.
The above projection $\varphi_D$ is coming from  the forgetful mapping $(A,\Theta,\iota)\mapsto (A,\Theta)$.
For detail, see \cite{Rotger} and \cite{Rotger2}.

\begin{rem}
The moduli spaces $\mathcal{S}_D$
are also called  Shimura curves in several works (for example, see Hashimoto-Murabayashi \cite{HM} or Yang \cite{Yang}). 
\end{rem}

\begin{rem}\label{RemRot}
The choice of the mapping 
$
 \varphi_D :  \mathbb{S}_D  \rightarrow  \mathcal{A}_2
$
is not unique 
and 
depends on  the principal polarizations on the corresponding abelian surfaces.
The number of the choice of $\varphi$ is calculated  by Rotger \cite{Rotger}.
Especially, if $D=6,10,15$,  then the choice of $\varphi_D$ is unique. 
Moreover, if $D=6,10,15$,  then $\varphi_D$ is a $4$ to $1$ mapping (\cite{Rotger2}). 
\end{rem}

\subsection{The result of Hashimoto}

In this subsection, we review the result of Hashimoto \cite{Hashimoto}.

Letting $B$ be an indefinite quaternion algebra with  discriminant $D=p_1\cdots p_t$
and $\mathfrak{O}$ be a maximal order of $B$,
take a prime number $p$ such that $p \equiv 5 $ $({\rm mod} \hspace{1mm} 8)$ and $\displaystyle \Big(\frac{p}{p_j} \Big)=-1$ for $p_j\not=2$.
Here, $\displaystyle \Big(\frac{a}{b}\Big)$ denotes the Legendre symbol.
We can assume that $B$ is expressed as
$
B=\mathbb{Q} + \mathbb{Q} i + \mathbb{Q} j + \mathbb{Q} ij
$
with
$
i^2=-D, j^2= p,  ij=-ji.
$
In this subsection, 
we consider the quaternion multiplication by $(\mathfrak{O},i_\rho)$ 
for  $\rho=i^{-1}=\displaystyle \frac{-i}{D}$.

\begin{rem}
We note  that $\rho$ is not always an element of $\mathfrak{O}$.
Nevertheless,
the notation in Definition \ref{QMDef} is available for $\rho=i^{-1}$.
In fact, 
$i\in \mathfrak{O}$ satisfies $i^2=-D$, $i\mathfrak{O} = \mathfrak{O} i$
and
$
\iota_i (\alpha )= i^{-1} \alpha'  i=  i \alpha' i^{-1}=\iota_{i^{-1}}(\alpha) $ $ (\alpha  \in B).
$
\end{rem}

Take $a,b\in\mathbb{Z}$ such that
$
a^2 D + 1 = p b.
$
We have 
a basis $\eta=\{\eta_1,\eta_2,\eta_3,\eta_4\}$ of $\mathfrak{O}$ 
is given by
$
 \eta_1= \frac{i +ij}{2} -\frac{p-1}{2} \frac{a D j + ij}{p}, \eta_2= -aD -\frac{a D j + ij}{p}, \eta_3=1, \eta_4=\frac{1+j}{2}.
$
We can see that
\begin{align}\label{symplecticB}
(E_\rho (\eta_j,\eta_k))=J=\begin{pmatrix} 0 & I_2 \\ -I_2 & 0 \end{pmatrix}.
\end{align}
For $\gamma\in \Gamma^{(1)}$, due to (\ref{norm1E}) and (\ref{symplecticB}),
$\eta_1 \gamma,\cdots,\eta_4\gamma$ gives another symplectic basis of $\mathfrak{O}$ with respect to $E_\rho$.
Hence,   there exists $M_\gamma \in Sp(4,\mathbb{Z})$ such that
$
(\eta_1\gamma,\cdots, \eta_4\gamma)=(\eta_1,\cdots,\eta_4) {}^tM_\gamma.
$
for any $\gamma\in\Gamma^{(1)}$.
For $w\in \mathbb{H}$, we set an $\mathbb{R}$-linear isomorphism
$f_w:B\otimes_\mathbb{Q} \mathbb{R}\simeq M_2(\mathbb{R}) \rightarrow \mathbb{C}^2$ given by
$
\alpha \mapsto \alpha \begin{pmatrix} w \\1 \end{pmatrix}.
$
Put $\omega_j =f_w(\eta_j) \in \mathbb{C}^2$ $(j=1,\cdots,4)$.
Then, $\Lambda_w=f_w(\mathfrak{O})=\langle  \omega_1,\cdots, \omega_4\rangle_\mathbb{Z}$ gives a lattice in $\mathbb{C}^2$
and $\mathbb{C}^2/\Lambda_w$ gives a complex torus with the period matrix $(\omega_1, \omega_2, \omega_3, \omega_4)=(\Omega_1(w)\Omega_2(w))$.
Then, $E_x:\mathbb{C}\times\mathbb{C} \rightarrow \mathbb{R}$ given by
$
E_w (f_w(\xi),f_w(\eta))=-E_\rho (\xi,\eta)
$
 induces a non-degenerate  skew symmetric pairing $E_w:\Lambda_w\times\Lambda_w \rightarrow \mathbb{Z}$. This gives an alternating Riemann form on the complex tours $\mathbb{C}^2/\Lambda_w$.
Therefore, we have the holomorphic embedding $\Omega$ in (\ref{Phi1}) given by
\begin{align}\label{HashimotoPhi}
w \mapsto \Omega(w)=\Omega_2^{-1}(w)\Omega_1(w) 
=\frac{1}{pw}
\begin{pmatrix}
\vspace{2mm} \displaystyle \overline{\varepsilon}^2+\frac{(p-1) a D }{2}w +D \varepsilon^2 w^2  & \displaystyle  \overline{\varepsilon} -(p-1)a D w - D\varepsilon w^2 \\
\displaystyle   \displaystyle  \overline{\varepsilon} -(p-1)a D w - D\varepsilon w^2  & \displaystyle -1 -2a D w + D w^2 
\end{pmatrix},
\end{align}
where $\varepsilon =  \frac{1+\sqrt{p}}{2}$
(see \cite{Hashimoto} Theorem 3.5).

Letting $pr$ be the canonical projection $\mathfrak{S}_2 \rightarrow \mathcal{A}_2$,
the image of $\mathbb{H}$ under the mapping $pr \circ \Omega$ corresponds to the Shimura curve $\mathcal{S}_D \subset \mathcal{A}_2$.

The matrix
$ \Omega (w)=\begin{pmatrix} \tau_1 & \tau_2 \\ \tau_2 & \tau_3 \end{pmatrix} \in \mathfrak{S}_2$ in (\ref{HashimotoPhi}) satisfies
the singular relations  in (\ref{singularR}). 
This is explicitly given by 
\begin{align}\label{HashimotoSing}
m \tau_1 +(m+2aD n) \tau_2 -\frac{p-1}{4} m \tau_3 + n(\tau_2^2 - \tau_1 \tau_3 ) + (a^2 D -b) D n=0
\end{align}
with two parameters $m,n \in \mathbb{Z}$ (\cite{Hashimoto} Theorem 5.1).
The invariant $\Delta$ in (\ref{singularR}) for the singular relation (\ref{HashimotoSing}) is given by the quadratic form
\begin{align}\label{quadraticD}
\Delta(m,n) = p m^2 + 4 a D  m n + 4 b  D n^2,
\end{align}
where $m,n\in \mathbb{Z}$.
Moreover, he showed the following theorem.

\begin{thm} \label{HashimotoC}
(1) (\cite{Hashimoto}, Theorem 5.2) For a positive non-square integer $\Delta$ such that $\Delta\equiv 1,0$ $({\rm mod} 4)$, the following conditions are equivalent:

(i) The number $\Delta$ is represented by the quadratic form $\Delta(m,n)$ in (\ref{quadraticD}) with relatively prime integers $m,n \in \mathbb{Z}$.

(ii) The image $\mathcal{S}_D=\varphi_D (\mathbb{S}_D) \subset \mathcal{A}_2$ of the Shimura curve is contained in the Humbert surface $\mathcal{H}_\Delta.$

(2) (\cite{Hashimoto}, Corollary 5.3) The Shimura curve  $\mathcal{S}_D$ is contained in the intersection $\mathcal{H}_{\Delta_1} \cap \mathcal{H}_{\Delta_2}$ of two Humbert surfaces if and only if $\Delta_1$ and $\Delta_2$ are given by $\Delta(m,n)$ with relatively prime integers $m,n$.  
\end{thm}

\begin{rem}
As we noted in Remark
\ref{RemRot},
the mapping $\varphi_D$  is not always unique.
For a generic $D$,
the mapping $\varphi_D$ and the model $\mathcal{S}_D$ via the embedding $\Omega$ in (\ref{HashimotoPhi}) 
depend on the triple $(p,a,b)$.
\end{rem}

\begin{exap}
For the case $D=6,$ we can take  $(p,a,b)=(5,2,5)$.  The  quadratic form is given by
\begin{align}\label{D=6}
\Delta_6(m,n)= 5 m^2 + 48 m n + 20 n^2.
\end{align}

For the case $D=10,$ we can take  $(p,a,b)=(13,3,7)$. The  quadratic form is given by
\begin{align}\label{D=10}
\Delta_{10}(m,n)= 13 m^2 + 120 m n + 280 n^2.
\end{align}

For the case $D=14,$ we can take  $(p,a,b)=(5,1,3)$. The  quadratic form is given by
\begin{align}\label{D=14}
\Delta_{14}(m,n)= 5 m^2 + 56 m n + 168 n^2.
\end{align}

For the case $D=15,$ we can take  $(p,a,b)=(53,22,137)$. The  quadratic form is given by
\begin{align}\label{D=15}
\Delta_{15}(m,n)= 53 m^2 + 1320 m n + 8220 n^2.
\end{align}

\end{exap}

Due to Theorem \ref{HashimotoC} (2), we have the following results.

\begin{exap}\label{Exap5,8}
The  image of the Shimura curves $\mathcal{S}_6$ attached to $(p,a,b)=(5,2,5)$ and $\mathcal{S}_{10}$  attached to $(p,a,b)=(13,3,7)$ 
are contained in the intersection $\mathcal{H}_5 \cap \mathcal{H}_8$ of the Humbert surfaces
because
$5=\Delta_6(5,-1)=\Delta_{10}(5,-1)$ and $8=\Delta_{6}(4,-1)=\Delta_{10}(4,-1)$.
\end{exap}

\begin{exap}\label{Exap5,12}
The  image of the Shimura curves $\mathcal{S}_6$ attached to $(p,a,b)=(5,2,5)$, 
 $\mathcal{S}_{14}$ attached to $(p,a,b)=(5,2,21)$ and  $\mathcal{S}_{15}$ attached to $(p,a,b)=(53,22,137)$.
are contained in the intersection $\mathcal{H}_5 \cap \mathcal{H}_{12}$ of the Humbert surfaces
because
$5=\Delta_6(5,-1)=\Delta_{15}(25,-2)=\Delta_{14}(1,0)$ and $12=\Delta_{6}(6,-1)=\Delta_{15}(12,-1)=\Delta_{14}(6,-1)$.
\end{exap}

\begin{exap}\label{Exap5,21}
The  image of the Shimura curves $\mathcal{S}_6$ attached to $(p,a,b)=(5,2,5)$ and $\mathcal{S}_{14}$ attached to $(p,a,b)=(5,1,3)$ 
are contained in the intersection $\mathcal{H}_5 \cap \mathcal{H}_{21}$ of the Humbert surfaces
because
$5=\Delta_6(5,-1)=\Delta_{15}(25,-2)=\Delta_{14}(1,0)$ and $21=\Delta_{6}(6,-1)=\Delta_{14}(7,-1)$.
However, the Shimura curve $\mathcal{S}_{15}$ is not contained in $\mathcal{H}_{5} \cap \mathcal{H}_{21}$ because $21$ is not represented by $\Delta_{15} (m,n)$ with $m,n\in \mathbb{Z}$.
\end{exap}

\section{The embedding $\mathcal{F} \hookrightarrow \mathcal{F}_{CD}$}

\subsection{The lattice polarized $K3$ surfaces}

A $K3$ surface $S$ is  a compact complex surface such that the canonical bundle $K_S=0$ and $H^1(S,\mathcal{O}_S)=0.$

By the canonical cup product and the Poincar\'e duality, 
the homology group $H_2(S,\mathbb{Z})$ has a lattice structure.
It is well known that the lattice $H_2(S,\mathbb{Z})$ is isometric to the even unimodular lattice $E_8(-1)\oplus E_8(-1) \oplus U \oplus U \oplus U$,
where $E_8(-1)$ is the negative definite even unimodular lattice of type $E_8$ and $U$ is the parabolic lattice of rank $2$.
Let ${\rm NS}(S)$ be the N\'eron-Severi lattice of a $K3$ surface $S$.
This is a sublattice of $H_2(S,\mathbb{Z}) $ generated by the divisors on $S$.
The orthogonal complement ${\rm Tr}(S)={\rm NS}(S)^\bot$ in $H_2(S,\mathbb{Z})$ is called the transcendental lattice of $S$.

Let $S$ be a $K3$ surface  and $M$ be  a lattice.
If we have a primitive lattice embedding $\iota : M \rightarrow {\rm NS}(S)$, the pair $(S,\iota)$ is called an $M$-polarized $K3$ surface.
Let $(S_1,\iota_1)$ and $(S_2,\iota_2)$ be $M$-polarized $K3$ surfaces.
If there exist an isomorphism $S_1 \rightarrow S_2$ of $K3$ surfaces such that $\iota_1 = f^* \circ \iota_2$,  $(S_1,\iota_1)$ and $(S_2,\iota_2)$ are isomorphic as $M$-polarized $K3 $ surfaces.
From now on, we often omit the primitive lattice embedding $\iota$.

\begin{rem}
When we consider  the moduli space and the period mapping of lattice polarized $K3$ surfaces,
we should pay attention to the ampleness of lattice polarized $K3$ surfaces (see \cite{Dolgachev}).
However, 
it is safe to apply the Torelli theorem to our cases.
See Remark \ref{RemM5Ample} and \ref{RemM0Ample}.
\end{rem}

\subsection{The icosahedral invariants}

Klein \cite{Klein} studied the action of the icosahedral group on $\mathbb{P}^2(\mathbb{C})={\rm Proj}(\mathbb{C}[\zeta_0,\zeta_1,\zeta_2])$.
He obtained a system of generators
$\mathfrak{A},\mathfrak{B},\mathfrak{C},\mathfrak{D}\in \mathbb{C}[\zeta_0,\zeta_1,\zeta_2]$
 of ring of icosahedral invariants.
The $\mathfrak{A}$ ($\mathfrak{B},\mathfrak{C},\mathfrak{D}$, resp.) is a homogeneous polynomial of weight $2$ ($6,10,15$, resp.)
and the ring is given by $
\mathbb{C}[\mathfrak{A}:\mathfrak{B}:\mathfrak{C}:\mathfrak{D}]/R (\mathfrak{A},\mathfrak{B},\mathfrak{C},\mathfrak{D}),
$
where $R (\mathfrak{A},\mathfrak{B},\mathfrak{C},\mathfrak{D})$ is the Klein's icosahedral relation
\begin{align}\label{KleinRel}
R (\mathfrak{A},\mathfrak{B},\mathfrak{C},\mathfrak{D}) =144\mathfrak{D}^2-(-1728\mathfrak{B}^5 +720\mathfrak{A}\mathfrak{C}\mathfrak{B}^3 -80 \mathfrak{A}^2 \mathfrak{C}^2 \mathfrak{B} +64\mathfrak{A}^3(5\mathfrak{B}^2-\mathfrak{A}\mathfrak{C})^2+\mathfrak{C}^3).
\end{align}

The Hilbert modular group $PSL(2,\mathcal{O}_5)$ acts on the product $\mathbb{H} \times \mathbb{H}$ of upper half planes.
We consider the symmetric Hilbert modular surface
$
\langle PSL(2,\mathcal{O}_5 ) , \tau \rangle \backslash (\mathbb{H}\times \mathbb{H})$,
 where $\tau$ is the involution given by $(z_1,z_2 )\mapsto (z_2 ,z_1)$. 

Hirzebruch obtained the following result.

\begin{prop} \label{KleinH} {\rm (\cite{Hirzebruch})}
(1)
The ring of  symmetric modular forms for  $PSL(2,\mathcal{O}_5)$ is isomorphic to the ring
$
\mathbb{C}[\mathfrak{A},\mathfrak{B},\mathfrak{C},\mathfrak{D}]/R(\mathfrak{A},\mathfrak{B},\mathfrak{C},\mathfrak{D}).
$
Here, $\mathfrak{A} $ {\rm(}$\mathfrak{B},\mathfrak{C},\mathfrak{D}$, resp.{\rm )} corresponds to a
 symmetric modular form for $PSL(2,\mathcal{O}_5)$ of weight $2$ {\rm (}$6,10,15$, resp.{\rm )}.

(2) 
The Hilbert modular surface $\langle PSL(2,\mathcal{O}_5 ) , \tau \rangle \backslash (\mathbb{H}\times \mathbb{H})$ has a compactification by adding one cusp 
$
(\mathfrak{A}:\mathfrak{B}:\mathfrak{C})=(1:0:0).
$
This compactification is the weighted projective plane
$\mathbb{P}(1:3:5)={\rm Proj} (\mathbb{C}[\mathfrak{A},\mathfrak{B},\mathfrak{C}])$.
\end{prop}

Set
\begin{align}\label{XY}
X=\frac{\mathfrak{B}}{\mathfrak{A}^3}, \quad \quad Y=\frac{\mathfrak{C}}{\mathfrak{A}^5}.
\end{align}
Then, the pair $(X,Y)$ gives an affine coordinate system of $\{\mathfrak{A}\not=0\}\subset \mathbb{P}(1:3:5)$.

\begin{rem}\label{HumbertHilbertRem}
The symmetric Hilbert modular surface $ \langle PSL(2,\mathcal{O}_5 ), \tau \rangle \backslash (\mathbb{H}\times \mathbb{H})$
coincides with the Humbert surface $\mathcal{H}_5$.
\end{rem}

\subsection{The family $\mathcal{F}=\{S(\mathfrak{A}:\mathfrak{B}:\mathfrak{C})\}$ of $K3$ surfaces}

In this subsection, we survey the results of \cite{NaganoTheta}.

For $(\mathfrak{A}:\mathfrak{B}:\mathfrak{C})  \in  \mathbb{P}(1:3:5)-\{ (1:0:0) \} $, 
we have the elliptic $K3$ surface 
\begin{align}\label{S(ABC)}
S(\mathfrak{A}:\mathfrak{B}:\mathfrak{C}):  z^2=x^3-4(4y^3 -5 \mathfrak{A} y^2) x^2 + 20 \mathfrak{B} y^3 x +\mathfrak{C} y^4.
\end{align}

The family $\mathcal{F}=\{S(\mathfrak{A}:\mathfrak{B}:\mathfrak{C}) | (\mathfrak{A}:\mathfrak{B}:\mathfrak{C})\in \mathbb{P}(1:3:5)-\{(1:0:0)\}\}$ is studied in \cite{NaganoTheta}.
By a detailed observation, we can prove  the following theorem.

\begin{prop}\label{K3Kiso} (\cite{NaganoTheta}, Section 2)
(1) For generic $(\mathfrak{A}:\mathfrak{B}:\mathfrak{C})  \in  \mathbb{P}(1:3:5)-\{ (1:0:0) \} $, the N\'eron-Severi lattice ${\rm NS}(S(\mathfrak{A}:\mathfrak{B}:\mathfrak{C}) )$ is given by the intersection matrix $E_8(-1) \oplus E_8(-1) \oplus \begin{pmatrix} 2 & 1 \\ 1 & -2 \end{pmatrix}=M_5$ and 
the transcendental lattice ${\rm Tr}(S(\mathfrak{A}:\mathfrak{B}:\mathfrak{C}) )$ is given by the intersection matrix $U \oplus \begin{pmatrix} 2 & 1 \\ 1 & -2 \end{pmatrix}=A_5.$

(2) The family $\mathcal{F}=\{S(\mathfrak{A}:\mathfrak{B}:\mathfrak{C}) \}$ gives the isomorphy classes of $M_5$-polarized $K3 $ surfaces.
Especially, $S(\mathfrak{A}_1:\mathfrak{B}_1:\mathfrak{C}_1)$ and $S(\mathfrak{A}_2:\mathfrak{B}_2:\mathfrak{C}_2)$ are isomorphic as $M_5$-polarized $K3$ surfaces if and only if $(\mathfrak{A}_1:\mathfrak{B}_1:\mathfrak{C}_1)=(\mathfrak{A}_2:\mathfrak{B}_2:\mathfrak{C}_2)$
in $\mathbb{P}(1:3:5)$.
\end{prop}

The period domain for the family $\mathcal{F}$ is given by the Hermitian symmetric space $\mathcal{D}=\{\xi \in \mathbb{P}^3(\mathbb{C}) | {}^t \xi A_5 \xi =0, {}^t \xi A_5 \overline{\xi}>0\}$ of type $IV.$ 
The space $\mathcal{D}$ has  two connected components $\mathcal{D}_+$ and $\mathcal{D}_-$.
We have the multivalued period mapping $\Phi:\mathbb{P}(1:3:5)-\{(1:0:0)\} \rightarrow \mathcal{D}_+$.
There exists a biholomorphic mapping $j: \mathbb{H} \times \mathbb{H} \rightarrow \mathcal{D}_+$.
Then, we have the multivalued mapping
$
j^{-1} \circ \Phi:\mathbb{P}(1:3:5)-\{(1:0:0)\} \rightarrow \mathbb{H} \times \mathbb{H},
$
that is given by
\begin{align}\label{ourperiod}
(\mathfrak{A}:\mathfrak{B}:\mathfrak{C}) \mapsto (z_1,z_2)=\Big(-\frac{ \int_{\Gamma_3}\omega + \frac{1-\sqrt{5}}{2}\int_{\Gamma_4}\omega}{ \int_{\Gamma_2} \omega},-\frac{ \int_{\Gamma_3}\omega +\frac{1+\sqrt{5}}{2}\int_{\Gamma_4} \omega}{ \int_{\Gamma_2}\omega}\Big),
\end{align}
where $\omega$ is the unique holomorphic $2$-form  up to a constant factor and $\Gamma_1,\cdots, \Gamma_4$ are certain $2$-cycles on  $S(\mathfrak{A}:\mathfrak{B}:\mathfrak{C}) $ (for detail, see \cite{NaganoTheta} and \cite{NaganoKummer}).

\begin{rem}\label{RemM5Ample}
In \cite{NaganoTheta},
the primitive lattice embedding $\iota:M_5\hookrightarrow {\rm NS}(S(\mathfrak{A}:\mathfrak{B}:\mathfrak{C}) )$
of the $M_5$-polarized $K3$ surfaces
is given  explicitly.
Especially, the image $\iota (M_5)$ is given by effective divisors of $S(\mathfrak{A}:\mathfrak{B}:\mathfrak{C}) $.
In fact, it assures an ampleness of lattice polarized $K3$ surfaces and  we can  apply the Torelli theorem  to our period mapping for $\mathcal{F}$ safely. 
For detailed argument, see \cite{NaganoTheta} Section 2.2.
\end{rem}

Let $\mathfrak{S}_2$ be the Siegel upper half plane consisting of $2\times 2$ complex matrices.
 The mapping $\mu_5:\mathbb{H}\times\mathbb{H} \rightarrow \mathfrak{S}_2$ given by
\begin{align}\label{Mullerembedding}
(z_1,z_2) \mapsto  \frac{1}{2\sqrt{5}}\begin{pmatrix}  (1+\sqrt{5})z_1 -(1-\sqrt{5})z_2 & 2(z_1-z_2) \\ 2(z_1 -z_2) & (-1+\sqrt{5}) z_1 +(1+\sqrt{5}) z_2 \end{pmatrix} 
\end{align}
gives 
 a modular embedding (see the following diagram).
\begin{align*}
\begin{CD}
\mathbb{H}\times \mathbb{H} @>\mu_5>> \mathfrak{S}_2\\
@V  \langle PSL(2,\mathcal{O}_5),\tau\rangle VV   @VV  Sp(4,\mathbb{Z})V \\
\mathbb{H}\times \mathbb{H}  @>>\mu_5> \mathfrak{S}_2\\
\end{CD}
\end{align*}
Moreover, $\mu_5 $ in (\ref{Mullerembedding}) gives a parametrization of the surface $N_5$ in Definition \ref{HumbertDf}:
\begin{align}\label{N5}
N_5=\{ \begin{pmatrix} \tau_1 & \tau_2 \\  \tau_2 &  \tau_3 \end{pmatrix} \in \mathfrak{S}_2 |  -\tau_1 + \tau_2 + \tau_3  =0\}.
\end{align}

For $\Omega\in\mathfrak{S}_2$ and $a,b \in \{0,1\}^2$ with ${}^t a b\equiv0 \hspace{1mm}({\rm mod}2)$,
set
$
\vartheta (\Omega;a,b)=\sum _{g\in \mathbb{Z}^2} {\rm exp}\Big(\pi \sqrt{-1} \big( {}^t\big( g+\frac{1}{2} a\big) \Omega \big( g+\frac{1}{2}a \big)+{}^tgb\big)\Big).
$
For $j\in\{0,1,\cdots,9\}$, we set
$
\theta_j (z_1,z_2)=\vartheta (\mu_5(z_1,z_2 ) ;a,b),
$
where the correspondence between $j$ and $(a,b)$ is given by Table 1.
\begin{table}
\center
{\footnotesize
\begin{tabular}{lcccccccccc}
\toprule
$j$&$0$&$ 1$ &$2$ &$3$&$4$&$5$&$6$&$7$&$8$&$9$  \\
\midrule
${}^t a$& $(0,0)$ &$(1,1)$ &$(0,0)$&$(1,1)$&$(0,1)$&$(1,0)$&$(0,0)$&$(1,0)$&$(0,0)$ &$(0,1)$ \\
${}^t b$ &$(0,0)$&$(0,0)$&$(1,1)$&$(1,1)$& $(0,0) $& $(0,0)$& $(0,1)$& $(0,1)$ & $(1,0)$&$(1,0)$\\
\bottomrule
\end{tabular}
}
\caption{The correspondence between $j$ and $(a,b)$.}
\end{table}
Let $a\in \mathbb{Z}$ and $j_1,\cdots,j_r \in\{0,\cdots,9\}$. 
We set  $\theta_{j_1,\cdots,j_r}^a  = \theta_{j_1}^a \cdots \theta_{j_r} ^a$.
The following $g_2$ ($s_6,s_{10}$, resp.) is   a symmetric Hilbert modular form of weight $2$ ($6,10,$ resp.) for  $\mathbb{Q}(\sqrt{5})$ (see M\"uller \cite{Muller}):
$
g_2= \theta_{0145}-\theta_{1279}-\theta_{3478}+\theta_{0268} +\theta_{3569},
s_6=2^{-8} (\theta_{012478}^2 +\theta_{012569}^2 +\theta_{034568}^2 + \theta_{236789}^2 +\theta_{134579}^2),
s_{10}=s_5^2 = 2^{-12} \theta_{0123456789}^2.$

\begin{prop}\label{ThetaRepn} (\cite{NaganoTheta} Theorem 4.1)
Using the  coordinates $(X,Y)$ of (\ref{XY}),
the inverse correspondence $(z_1,z_2)\mapsto (X(z_1,z_2),Y(z_1,z_2))$ of the period mapping (\ref{ourperiod}) for $\mathcal{F}$ has the following theta expression
\begin{align}\label{Ourtheta}
\displaystyle X(z_1,z_2)=2^5 \cdot 5^2 \cdot \frac{s_6(z_1,z_2)}{g_2^3(z_1,z_2)},
\quad \quad \quad 
\displaystyle Y(z_1,z_2)=2^{10} \cdot 5^5  \cdot\frac{s_{10}(z_1,z_2)}{g_2^5 (z_1,z_2)}.
\end{align}
Moreover, $X$ and $Y$ give a system of generators of the field of symmetric Hilbert modular functions for $\mathbb{Q}(\sqrt{5})$.
\end{prop}

We call the divisor  
$
\square=\{(z_1,z_2)\in\mathbb{H}\times \mathbb{H}| z_1=z_2\}
$
the diagonal.
On the diagonal $\square$, it holds
\begin{align}\label{XYDiag}
X(z,z)=\frac{25}{27} \frac{1}{J(z)}, \quad \quad \quad   Y(z,z)=0  .
\end{align}

\subsection{The family $\mathcal{F}_{CD}=\{S_{CD} (\alpha,\beta,\gamma,\delta)\}$ of $K3$ surfaces}

In \cite{CD} and \cite{NaganoShiga}, the family $\mathcal{F}_{CD}=\{S_{CD} (\alpha,\beta,\gamma,\delta)|(\alpha:\beta:\gamma:\delta)  \in  \mathbb{P}(2:3:5:6)-\{ \gamma=\delta=0 \} \}$ of $K3$ surfaces is studied in detail, 
where
\begin{align}\label{Mother}
S_{CD}(\alpha:\beta:\gamma:\delta): y^2 =x^3 +(-3 \alpha t^4 - \gamma t^5)x +(t^5 -2\beta t^6 + \delta t^7).
\end{align}
The defining equation (\ref{Mother}) gives the structure of an elliptic surface $(x,y,t) \mapsto t$ with
the  singular fibres  $II^*+5 I_1 + III^*$. 
From this elliptic fibration,
we can obtain a marked $K3$ surface and prove the following theorem.

\begin{rem}
In this paper, we use  the notation $\mathcal{F}_{CD}$ and $S_{CD}$,
since the surface (\ref{Mother}) appears in the paper of Clingher and Doran \cite{CD}. 
On the other hand,
Kumar \cite{Kumar} independently studied elliptic $K3$ surfaces with the same singular fibres.
So, we need to recall  Kumar's contribution.
\end{rem}

\begin{prop}(\cite{NaganoShiga}, Section 2 and 3) \label{CDNS}
(1) If an elliptic $K3$ surface $S$ with the elliptic fibration $(x,y,t)\mapsto t$  has the singular fibres  of type $II^* $ at $t=0$,   $III^*$ at $t=\infty$ and other five fibres  of type $I_1$,
then $S$ is given by the Weierstrass equation  in  (\ref{Mother}).

(2) For generic $(\alpha:\beta:\gamma:\delta)  \in  \mathbb{P}(2:3:5:6)-\{ \gamma=\delta=0 \} $, the N\'eron-Severi lattice ${\rm NS}(S_{CD}(\alpha:\beta:\gamma:\delta)$ is given by the intersection matrix $E_8(-1) \oplus E_7(-1) \oplus U=M_0$ and 
the transcendental lattice ${\rm Tr}(S_{CD}(\alpha:\beta:\gamma:\delta))$ is given by the intersection matrix $U \oplus U \oplus \langle -2 \rangle=A_0.$

(3) The family $\mathcal{F}_{CD}=\{S_{CD}(\alpha:\beta:\gamma:\delta)\}$ gives the  isomorphism classes of $M_0$-polarized $K3 $ surfaces.
Especially, $S_{CD}(\alpha_1:\beta_1:\gamma_1:\delta_1)$ and $S_{CD}(\alpha_2:\beta_2:\gamma_2:\delta_2)$ are isomorphic as $M_0$-polarized $K3$ surfaces if and only if $(\alpha_1:\beta_1:\gamma_1:\delta_1)=(\alpha_2:\beta_2:\gamma_2:\delta_2)$
in $\mathbb{P}(2:3:5:6)$.

\end{prop}

Let $\mathcal{D}_{0} =\{\xi\in \mathbb{P}^4(\mathbb{C}) |   {}^t \xi A_0 \xi =0, {}^t \xi A_0 \overline{\xi} >0\} $.
The period domain for the family $\mathcal{F}_{CD}$ is given by the quotient space 
$ PO(A_0,\mathbb{Z}) \backslash \mathcal{D}_{0} $,
where $PO(A_0,\mathbb{Z})=\{\ M\in GL(4,\mathbb{Z}) | {}^t M A_0 M=A_0 \}$.
In fact, 
there exists a holomorphic  mapping 
$\mathcal{D}_0 \rightarrow \mathfrak{S}_2$
such that this mapping induces the isomorphism  
$ PO(A_0,\mathbb{Z}) \backslash \mathcal{D}_{0}  \simeq Sp(4,\mathbb{Z}) \backslash \mathfrak{S}_2 =\mathcal{A}_2$.
The transcendental lattice ${\rm Tr}(S_{CD}(\alpha:\beta:\gamma:\delta))$ is Hodge isometric to the transcendental lattice ${\rm Tr}(A)$ of  a generic principally polarized abelian surface 
and the family $\mathcal{F}_{CD}$ gives the same variations of Hodge structures of weight $2$  with  the family of  principally polarized abelian surfaces (see \cite{NaganoShiga} Section 3).

\begin{rem}\label{RemM0Ample}
In \cite{NaganoShiga} Section3,
the primitive lattice embedding $\iota:M_0 \hookrightarrow {\rm NS}(S_{CD}(\alpha:\beta:\gamma:\delta))$ of $M_0$-polarized $K3$ surfaces is attained by taking appropriate effective divisors of $S_{CD}(\alpha:\beta:\gamma:\delta)$.
This argument guarantees an ampleness of lattice polarized $K3$ surfaces and 
 it is safe to apply the Torelli theorem for lattice polarized $K3$ surface to our family $\mathcal{F}_{CD}$.
\end{rem}

Let $\mathcal{M}_2$ be the moduli space of genus $2$ curves.
Let $\mathbb{P}(1:2:3:5)=\{(\zeta_1:\zeta_2:\zeta_3:\zeta_5)\}$ be the weighted projective space.
It is well-known that 
$
\mathcal{M}_2=\mathbb{P}(1:2:3:5) -\{\zeta_5=0\}.
$
In fact,
by the Igusa-Clebsch invariants $I_2,I_4,I_6,I_{10}$ of degree $2,4,6,10$ for a genus $2$ curve,
$(I_2:I_4:I_6:I_{10})$ gives a well-defined point of the moduli space $\mathcal{M}_2$.
We note that the moduli space $\mathcal{M}_2$ is a Zariski open set of the moduli space $\mathcal{A}_2$ of principally polarized abelian surfaces ($\mathcal{M}_2$ is the complement of the divisor given by  the points corresponding to the product of elliptic curves).

By a study of elliptic $K3$ surfaces, we can prove the following proposition.

\begin{prop}
(\cite{Kumar}, \cite{CD} see also \cite{NaganoShiga})
The  point 
$(I_2:I_4:I_6:I_{10})\in \mathcal{M}_2= \mathbb{P}(1:2:3:5)-\{\zeta_5=0\}$ 
corresponds to the point 
$(\alpha:\beta:\gamma:\delta)\in \mathbb{P}(2:3:5:6)-\{\gamma=\delta=0\}$
of the moduli space of $M_0$-polarized $K3$ surfaces
by the following birational transformation:
\begin{align}\label{abcd}
 \alpha=\frac{1}{9} I_4 ,\quad
  \beta=\frac{1}{27} (-I_2 I_4 +3 I_6),\quad
 \gamma= 8 I_{10},  \quad
 \delta=\frac{2}{3} I_2 I_{10}.
\end{align}
\end{prop}

The Humbert surface $\mathcal{H}_5$ 
is a subvariety of the moduli space $\mathcal{A}_2$.
Hence, the defining equation of  $\mathcal{H}_5$ can be described by the equation in $(\alpha:\beta:\gamma:\delta)\in\mathbb{P}(2:3:5:6).$
By an observation of 
the elliptic fibration given by (\ref{Mother}),
we can prove the following theorem.
Especially, the equation (\ref{modular5})  shall give the defining equation of $\mathcal{H}_5$.

\begin{prop}(\cite{NaganoShiga} Theorem 4.4 ) \label{NS4.4}

(1)
If and only if the equation
\begin{align}\label{modular5}
(-\alpha^3 -\beta^2 +\delta)^2- 4 \alpha (\alpha \beta -\gamma)^2=0
\end{align}
holds,
there exists a non trivial section $s_5$ of  $\{S_{CD}(\alpha:\beta:\gamma:\delta)\}$ as illustrated in Figure 1.

(2) 
If the modular equation (\ref{modular5}) holds, 
the N\'eron-Severi lattice of the $K3$ surface $S _{CD}(\alpha:\beta:\gamma:\delta)$ is generically given by the intersection matrix
$
M_5.$
\end{prop}

We call the equation (\ref{modular5}) the modular equation for $\Delta=5$.

\begin{figure}[h]
\center
\includegraphics[scale=0.6]{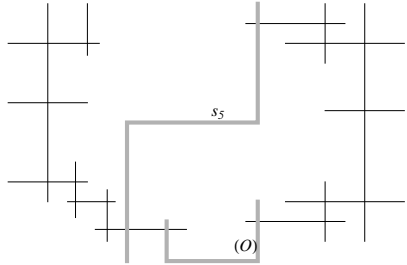}
\caption{The section $s_5$.}
\end{figure}

\begin{rem}\label{RemarkModular8}
The modular equation (\ref{modular5}) for  $\Delta=5$ is much simpler than  modular equations  for other discriminants. 
For example,
the modular equation for  $\Delta=8$ is given by
{\footnotesize
\begin{align}\label{Mod8}
\notag &1024 \alpha^{15} - 5120 \alpha^{12} \beta^2 + 10240 \alpha^9 \beta^4 - 10240 \alpha^6 \beta^6 + 
 5120 \alpha^3 \beta^8 - 1024 \beta^{10} - 941056 \alpha^{11} \beta \gamma\\
 \notag & + 
 1053696 \alpha^8 \beta^3 \gamma + 715776 \alpha^5 \beta^5 \gamma - 828416 \alpha^2 \beta^7 \gamma - 
 7556464 \alpha^10 \gamma^2 + 131492384 \alpha^7 \beta^2 \gamma^2 \\
\notag  &+ 
 39076880 \alpha^4 \beta^4 \gamma^2 + 13934400 \alpha \beta^6 \gamma^2 + 
 1491324088 \alpha^6 \beta \gamma^3 - 918848440 \alpha^3 \beta^3 \gamma^3 \\
 \notag &  - 
 36968000 \beta^5 \gamma^3 + 13611473901 \alpha^5 \gamma^4 - 
 718342500 \alpha^2 \beta^2 \gamma^4 + 9079601250 \alpha \beta \gamma^5 + 7737809375 \gamma^6 \\
\notag  &- 
 343808 \alpha^12 \delta - 647168 \alpha^9 \beta^2 \delta + 2234880 \alpha^6 \beta^4 \delta - 
 1153024 \alpha^3 \beta^6 \delta - 90880 \beta^8 \delta + 2442144 \alpha^8 \beta \gamma \delta \\
\notag &- 
 86206272 \alpha^5 \beta^3 \gamma \delta + 12985248 \alpha^2 \beta^5 \gamma \delta - 
 1669045416 \alpha^7 \gamma^2 \delta - 1449171160 \alpha^4 \beta^2 \gamma^2 \delta \\
\notag &+ 
 268484800 \alpha \beta^4 \gamma^2 \delta - 157452560 \alpha^3 \beta \gamma^3 \delta - 
 772939000 \beta^3 \gamma^3 \delta - 15745060125 \alpha^2 \gamma^4 \delta \\
\notag &+ 
 29370256 \alpha^9 \delta^2 - 56832480 \alpha^6 \beta^2 \delta^2 + 
 37166352 \alpha^3 \beta^4 \delta^2 - 2626240 \beta^6 \delta^2 + 
 1230170496 \alpha^5 \beta \gamma \delta^2 \\
\notag &
 - 155485248 \alpha^2 \beta^3 \gamma \delta^2 
- 
 27876720 \alpha^4 \gamma^2 \delta^2 + 2388102200 \alpha \beta^2 \gamma^2 \delta^2 - 
 2315093000 \beta \gamma^3 \delta^2 - 86058160 \alpha^6 \delta^3 \\
 &\notag
 + 
 3605888 \alpha^3 \beta^2 \delta^3 
  - 22815760 \beta^4 \delta^3 - 
 1231671584 \alpha^2 \beta \gamma \delta^3 + 1704478600 \alpha \gamma^2 \delta^3 + 
 85375664 \alpha^3 \delta^4\\
 & + 53878880 \beta^2 \delta^4 - 28344976 \delta^5=0.
\end{align}
}
We can check that for a generic $(\alpha,\beta,\gamma,\delta)$ satisfying (\ref{Mod8}),
the transcendental lattice is given by $U\oplus \begin{pmatrix} 2 & 2 \\ 2 & -2 \end{pmatrix}$. 
So, (\ref{Mod8}) gives  a counterpart  of the equation (\ref{modular5}).
\end{rem}

\subsection{The embedding $\Psi_5: \mathbb{P}(1:3:5) \hookrightarrow \mathbb{P}(2:3:5:6)$}

The family of $M_5$-polarized $K3$ surfaces is a subfamily of the family of $M_0$-polarized $K3$ surface.
Therefore, Proposition \ref{K3Kiso}, \ref{CDNS} and \ref{NS4.4} imply that the family $\mathcal{F}=\{S(\mathfrak{A}:\mathfrak{B}:\mathfrak{C}) \}$  is a subfamily of the family $\mathcal{F}_{CD}=\{S_{CD}(\alpha:\beta:\gamma:\delta)\}$. 
Moreover, together with Remark \ref{HumbertHilbertRem},
the modular equation (\ref{modular5}) gives the defining equation of the Humbert surface $\mathcal{H}_5$.
In this subsection, 
we realize the embedding $\mathcal{F}\hookrightarrow \mathcal{F}_{CD}$ explicitly.
This is given by the embedding $\Psi_5: \mathbb{P}(1:3:5) \hookrightarrow \mathbb{P}(2:3:5:6) $
of varieties.

Since the modular equation (\ref{modular5}) for $\Delta=5$ is very simple and the coordinates $\mathfrak{A},\mathfrak{B}$ and $\mathfrak{C}$
have the explicit theta expressions (\ref{Ourtheta}) via the period mapping  for the family $\mathcal{F}$,  
we can study the pull-back $\Psi_5^* (V)$ of a variety $V\subset \mathbb{P}(2:3:5:6)$ quite effectively.
Especially,
in Section 4 and 5,
we shall consider the pull-back $\Psi_5^*(\mathcal{S}_D)$ of the Shimura curve $\mathcal{S}_D $ for $D=6,10,14$ and $15$ in $\mathbb{P}(2:3:5:6)$.

\begin{lem}\label{LemmaEmb5}
The elliptic $K3$ surface $S(\mathfrak{A}:\mathfrak{B}:\mathfrak{C}) $ is birationally equivalent to the elliptic $K3$ surface 
\begin{align}\label{II^*III^*5}
y_0^2= s^3 + \Big( -\frac{25}{12} \mathfrak{A}^2 u^4 -\frac{1}{32} \mathfrak{C} u^5  \Big)s +\Big(u^5 +(\frac{125}{108} \mathfrak{A}^3 -\frac{5}{4} \mathfrak{B})u^6 + (\frac{25}{64} \mathfrak{B}^2 -\frac{5}{96} \mathfrak{A}\mathfrak{C})u^7  \Big).
\end{align}
The elliptic  fibration $\pi:(s,u,y_0) \mapsto u$ given by (\ref{II^*III^*5}) has singular fibres  
$\pi^{-1}(0)$ of type $II^*$, $\pi^{-1}(\infty)$ of type $III^*$ and other 
five singular fibres of type $ I_1$.
\end{lem}

\begin{proof}
First, by the correspondence 
$$
x=\frac{x_1}{16 t}, \quad  y=-\frac{x_1}{16 t^2}, \quad  z=\frac{x_1 y_1}{256 t^4},
$$
the surface $S(\mathfrak{A}:\mathfrak{B}:\mathfrak{C}) $ in (\ref{S(ABC)}) is transformed to
\begin{align}\label{DoubleCoverKummer}
y_1^2=x_1 (x_1^2 +(20  \mathfrak{A} t^2 -20 \mathfrak{B} t + \mathfrak{C})x + 16 t^5).
\end{align}
The elliptic surface given by (\ref{DoubleCoverKummer}) has the singular fibres of type $I_{10}+5 I_1 + III^*$.

Next, by the birational transformation
\begin{align}\label{2neighbor}
x_1=\frac{\mathfrak{C}^4 u_0^5}{x_2^4}, \quad y_1= \frac{\sqrt{\mathfrak{C}^9} u_0^5}{x_2^6}  y_2 , \quad t=\frac{\mathfrak{C} u_0}{x_2}, 
\end{align}
we have
\begin{align}\label{3.17}
y_2^2= x_2^4 + (16-20 \mathfrak{B} u_0) x_2^3 + 20 \mathfrak{A}\mathfrak{C} u_0^2 x_2 + \mathfrak{C} u_0^5. 
\end{align}
The equation (\ref{3.17}) gives a double covering of a polynomial of degree $4$ in $x_2$.
According to Section 3.1 of \cite{AKMMMP},
such a polynomial can be transformed to a Weierstrass equation.
In our case,
putting 
\begin{align*}
\begin{cases}
\vspace{2mm}
& \displaystyle x_2= \frac{6 s_1 (-4 + 5 \mathfrak{B} u_0) - 60 \mathfrak{A} \mathfrak{C} u_0^2 (-4 + 5 \mathfrak{B} u_0 ) + \sqrt{6} y_3}{6 (-96 + s_1 + 240 \mathfrak{B} u _0- 150 \mathfrak{B}^2 u_0^2 + 20 \mathfrak{A} \mathfrak{C} u_0^2)},\\
\vspace{2mm}
&\displaystyle y_2=  \frac{ 48 + s_1 - 120 \mathfrak{B} u_0 + 75 \mathfrak{B}^2 u_0^2 - 10 \mathfrak{A} \mathfrak{C} u_0^2}{3} \\
&\hspace{2cm}  \displaystyle - \frac{(27 (2 (-4 + 5 \mathfrak{B} u_0) (16 - 40 \mathfrak{B} u_0 + 25 \mathfrak{B}^2 u_0^2 - 5 \mathfrak{A} \mathfrak{C} u_0^2) +  \frac{y_3}{3 \sqrt{6}})^2)}{3 (-96 + s_1 + 240 \mathfrak{B} u_0 - 150 \mathfrak{B}^2 u_0^2 + 20 \mathfrak{A} \mathfrak{C} u_0^2)^2},
\end{cases}
\end{align*}
we have the Weierstrass equation 
\begin{align}\label{3.18}
y_3^2= \text{polynomial in } s_1 \text{ of degree } 3.
\end{align}
Put
$$
s_1=3 \mathfrak{C} s, \quad y_3=\sqrt{27 \mathfrak{C}^3}y_0, \quad u_0=\frac{u}{2},  
$$
to  (\ref{3.18}).
Then, 
we have (\ref{II^*III^*5}).
\end{proof}

\begin{rem}
The transformation in (\ref{2neighbor}) gives an example of 2-neighbor step,
that is a method  to find a new elliptic fibration. 
By (\ref{2neighbor}),
we have
$
\displaystyle  u_0=\frac{x}{t^4}.
$
The new parameter $u_0$ has a pole of order $4$ at $t=0$.
This implies that we have a singular fibre of type $III^*$ at $u_0=\infty$
(Figure 2).
\begin{figure}[h]
\center
\includegraphics[scale=0.6]{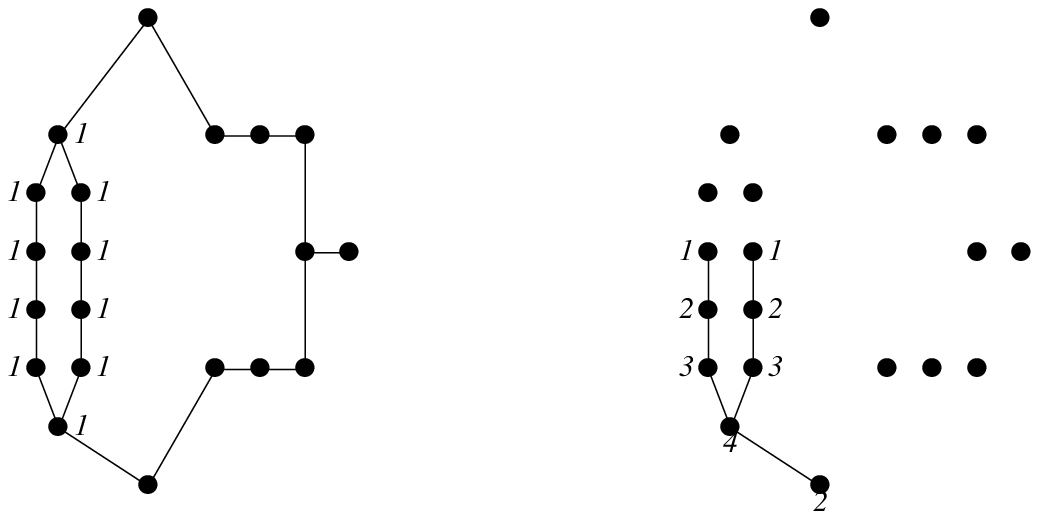}
\caption{2-neiborgh step in (\ref{2neighbor})}
\end{figure}
\end{rem}

\begin{rem}
The section $s_5$ in Proposition \ref{NS4.4} (1) has the explicit form
\begin{align}\label{section555}
t & \mapsto (x(t),y(t),t)  \notag \\
&=\Big( \frac{32}{Y} + \frac{40  X}{Y} t+ \frac{5  (15 X^2 - 2 Y)}{6 Y} t^2 , \frac{128 \sqrt{2}}{Y^{3/2}} + \frac{240 \sqrt{2}  X}{Y^{3/2}} t  - \frac{10 \sqrt{2}  (-15 X^2 + Y)}{Y^{3/2}} t^2 - \frac{25  (-5 X^3 + X Y)}{2 \sqrt{2} Y^{3/2}} t^3 ,t \Big)
\end{align}
\end{rem}

\begin{thm}\label{PsiThm}
The point $(\alpha:\beta:\gamma:\delta)\in \mathbb{P}(2:3:5:6)$ satisfies the modular equation (\ref{modular5})
if and only if the point $(\alpha:\beta:\gamma:\delta)$
is in the image of the 
embedding $\Psi_5 : \mathbb{P}(1:3:5) \rightarrow  \mathbb{P}(2:3:5:6) $ given by
\begin{align*}
(\mathfrak{A}:\mathfrak{B}:\mathfrak{C})  \mapsto  (\alpha:\beta:\gamma:\delta)   =(\alpha_5(\mathfrak{A}:\mathfrak{B}:\mathfrak{C}):\beta_5(\mathfrak{A}:\mathfrak{B}:\mathfrak{C}):\gamma_5(\mathfrak{A}:\mathfrak{B}:\mathfrak{C}):\delta_5(\mathfrak{A}:\mathfrak{B}:\mathfrak{C})),
\end{align*}
where
\begin{align}\label{(ABCDXY)}
\begin{cases}
\vspace{2mm}
&\alpha_5(\mathfrak{A}:\mathfrak{B}:\mathfrak{C})=\displaystyle \frac{25}{36} \mathfrak{A}^2,\\
\vspace{2mm}
&\beta_5 (\mathfrak{A}:\mathfrak{B}:\mathfrak{C})= \displaystyle  \frac{1}{2}\Big(- \frac{125}{108} \mathfrak{A}^3 + \frac{5}{4} \mathfrak{B} \Big),\\
\vspace{2mm}
&\gamma_5 (\mathfrak{A}:\mathfrak{B}:\mathfrak{C})= \displaystyle  \frac{1}{32} \mathfrak{C} , \\
&\delta_5 (\mathfrak{A}:\mathfrak{B}:\mathfrak{C})= \displaystyle \frac{25}{64} \mathfrak{B}^2 - \frac{5}{96} \mathfrak{A} \mathfrak{C}.
\end{cases}
\end{align}
\end{thm}

\begin{proof}
According to Proposition \ref{NS4.4},
if $(\alpha:\beta:\gamma:\delta)\in \mathbb{P}(2:3:5:6)$ satisfies the modular equation (\ref{modular5}),
then the N\'eron-Severi lattice ${\rm NS}(S(\alpha:\beta:\gamma:\delta))$ is generically given by the intersection matrix $ M_5$. 

On the other hand, a family of the isomorphism classes of $M_5$-marked $K3$ surfaces is given by $\mathcal{F}$.
By Lemma \ref {LemmaEmb5}, $S(\mathfrak{A}:\mathfrak{B}:\mathfrak{C}) $ is birationally equivalent to the surface given by (\ref{II^*III^*5}) with the section (\ref{section555}).
Therefore,
the Weierstrass  equation (\ref{II^*III^*5}) induces an embedding $\mathcal{F} \hookrightarrow \mathcal{F}_{CD}$ of the family of elliptic  $K3$ surfaces
with the singular fibers of type $II^* +  5 I_1 + III^*$.
Together with Proposition \ref{CDNS}, we have 
(\ref{(ABCDXY)}) by comparing the coefficients of  (\ref{Mother}) and (\ref{II^*III^*5}). 
\end{proof}

\begin{rem}
We can check that
any point of the image $\Psi_5(\mathbb{P}(1:3:5))$ 
 satisfies  the modular equation
(\ref{modular5}).
\end{rem}

\section{The family $\mathcal{F}_{EK}^\Delta$ of $K3$ surfaces}

Elkies and Kumar \cite{EK} obtained  rational models of Hilbert modular surfaces.
Especially,
in their argument,
they used  parametrizations of the Humbert surfaces $\mathcal{H}_\Delta$ for  fundamental discriminants $\Delta$ such that $1<\Delta<100$.
Their method was the  following. 
They consider a family,
that is called $\mathcal{F}_{EK}^\Delta$ in this paper,
 of elliptic $K3$ surfaces with two complex parameters.
A generic member of $\mathcal{F}_{EK}^\Delta$ has a suitable transcendental lattice and the moduli space of $\mathcal{F}_{EK}^\Delta$
is birationally equivalent to the Humbert surface $\mathcal{H}_\Delta$.
Moreover, the family $\mathcal{F}_{EK}^\Delta$ can be regarded as a subfamily of $\mathcal{F}_{CD}$. 
It follows that the two complex parameters of $\mathcal{F}_{CD}$ give a parametrization $\chi_\Delta$ of  $\mathcal{H}_\Delta$.

In this paper, 
we shall use the  parametrization $\chi_8$ ($\chi_{12},\chi_{21}$, resp.)  for the Humbert surface $\mathcal{H}_8$ ($\mathcal{H}_{12},\mathcal{H}_{21}$, resp.).
We survey their results in this section.

However, we remark that
the explicit forms of the parametrization $\chi_\Delta$ appeared  in the paper \cite{EK} only for the case $\Delta=5$ and $8$.
Then, we need to calculate the explicit forms of $\chi_{12}$ and $\chi_{21}$ from the families $\mathcal{F}_{EK}^{12}$ and $\mathcal{F}_{EK}^{21}$
(see Section 3.3 and 3.4).

\begin{rem}\label{Remchi}
The choice of the parametrization of the Humbert surface $\mathcal{H}_\Delta$ is not unique.
In fact, the parametrization $\chi_\Delta$ due to Elkies and Kumar
depends on the choice of 
an elliptic fibration of a generic member of $\mathcal{F}_{EK}^\Delta$.
To the best of the author's knowledge,
it is not easy to study modular properties of the parametrization $\chi_\Delta$.
For example, 
it seems highly non trivial problem 
to obtain an explicit expression of the parametrization of $\chi_\Delta$ via the Hilbert modular forms for $\mathbb{Q}(\sqrt{\Delta})$. 
See Remark \ref{RemPsi} also.
\end{rem}

\subsection{The case of discriminant $5$}

Before we consider the cases $\Delta=8,12,21$,
let us see the case $\Delta=5$.

In  Section 6 of \cite{EK},
the family $\mathcal{F}_{EK}^5$ of $K3 $ surfaces  is studied.
A generic member of $\mathcal{F}_{EK}^5$ is 
given by the defining equation
$$
y^2=x^3+\frac{1}{4} t^3 (-3g^2 t +4 ) x -\frac{1}{4} t^5 (4 h^2 t^2 +(4 h +g^3 ) t + (4 g+1)),
$$
where $g$ and $h$ are two complex parameters.

Using this family,
Elkies and Kumar obtained a parametrization of the Humbert  surface $\mathcal{H}_5$.
In \cite{EK}, $\mathcal{H}_5$ is realized  as a surface in
the moduli space $\mathcal{M}_2={\rm Proj }(\mathbb{C}[I_2, I_4, I_6, I_{10}]) $ by the parametrization
\begin{align*}
%\begin{cases}
 I_2=6(4 g+1), \quad
  I_4=9g^2, \quad
I_6= 9( 4 h  + 9g^3 +2g^2 ),\quad
 I_{10}=4 h^2.
%\end{cases}
\end{align*}
Together with (\ref{abcd}),  
we obtain the mapping $\chi_5 : \mathbb{P}(1:2:5) \rightarrow \mathbb{P}(2:3:5:6)$ 
given by $(k:g:h) \mapsto ( \alpha'_5 (k:g:h):\beta'_5 (k:g:h):\gamma'_5 (k:g:h):\delta'_5 (k:g:h))$,
where
\begin{align*}
 \alpha'_5 (k:g:h) =g^2, \quad
 \beta_5' (k:g:h) = g^3 + 4h k, \quad
 \gamma_5' (k:g:h) = 32 h^2, \quad
 \delta_5' (k:g:h) = 16 h^2 (4 g + k^2).
\end{align*}
The mapping $\chi_5$ gives a parametrization of the Humbert surface $\mathcal{H}_5$.
Any point of the image of $\chi_5$ satisfies the modular equation (\ref{modular5}).

\begin{rem}\label{RemPsi}
The above $\chi_5$ is a parametrization different from $\Psi_5:\mathbb{P}(1:3:5) \hookrightarrow \mathbb{P}(2:3:5:6)$ in Theorem \ref{PsiThm}.
In Section 4 and 5,
we shall use only $\Psi_5$.
The parametrization  $\Psi_5$ 
has good modular properties 
and 
is more convenient than $\chi_5$ for our purpose.
For example,
the weighted projective space  $\mathbb{P}(1:3:5)={\rm Proj}(\mathbb{C}[\mathfrak{A},\mathfrak{B},\mathfrak{C}])$
is a canonical compactification of the Hilbert modular surface (see Section 3.2)
and the coordinates $\mathfrak{A},\mathfrak{B}$ and $\mathfrak{C}$ have an expression  by  Hilbert modular forms (see Proposition \ref{ThetaRepn}).
\end{rem}

\subsection{The case of discriminant  $8$}

In  Section 7 of \cite{EK},
the family $\mathcal{F}_{EK}^8$ of $K3 $ surfaces  is studied.
A generic member of $\mathcal{F}_{EK}^8$ is 
given by the defining equation
$$
y^2=x^3+ t ((2 r +1) t + r ) x^2 +2 r s  t^4 (t+1) x + r s t^7
$$
where $r$ and $s$ are two complex parameters.
This  Weierstrass equation
gives the elliptic fibration $(x,y,t) \mapsto t$ with the singular fibres of type $III^*$ and $I_5^*$.
A generic member of $\mathcal{F}_{EK}^8$ admits  another elliptic fibration  with  the singular fibres of type  $II^* + 5 I_1+III^*$ 
and we regard $\mathcal{F}_{EK}^8$ as a subfamily of $\mathcal{F}_{CD}$ (see Proposition \ref{CDNS}).
Thus, Elkies and Kumar gave the  correspondence
\begin{align*}
\begin{cases}
& I_2=-4(3s + 8r -2), \quad
  I_4=4(9r s + 4 r^2 + 4 r + 1), \\
& I_6= -4( 36 r s^2 + 94 r^2 s -35 r s + 4 s + 48 r^3 + 40 r^2 + 4 r -2 ), \quad
 I_{10}=-8 s^2 r^3.
\end{cases}
\end{align*}
Together with (\ref{abcd}),
we have the following correspondence
$
\chi_{8}:\mathbb{P}^2(\mathbb{C}) \rightarrow \mathbb{P}(2:3:5:6) 
$
given by
$(q:r:s) \mapsto (\alpha_8(r:s:q):\beta_8(r:s:q):\gamma_8(r:s:q):\delta_8(r:s:q))$
where
\begin{align}\label{parameter8}
\begin{cases}
&\alpha_8(r:s:q)=\frac{4}{9} (q^2 + 4 q r + 4 r^2 + 9 r s),\\
&\beta_8 (r:s:q) = \frac{1}{27} \Big(16 (-2 q + 8 r + 3 s) (q^2 + 4 q r + 4 r^2 + 9 r s)  \\
&\hspace{3cm}   -12 (-2 q^3 + 4 q^2 r + 40 q r^2 + 48 r^3 + 4 q^2 s - 35 q r s + 
      94 r^2 s + 36 r s^2)\Big),\\
     &\gamma_8(r:s:q)=-64 r^3 s^2,\quad
 \delta_8(r:s:q)=  \frac{64}{3} r^3 s^2 (-2 q + 8 r + 3 s).
\end{cases}
\end{align}

The correspondence $\chi_8$ gives a parametrization of the Humbert surface $\mathcal{H}_8$.
We shall use $\chi_8$ in Section 4.

\subsection{The case of discriminant $12$}

In \cite{EK} section 8,
Elkies and Kumar considered the elliptic $K3$ surfaces given by the following equation
\begin{align}\label{EK12}
y^2=x^3 + ((1 - f^2) (1 - t) + t) t x^2 + 2 e t^3 (t - 1) x + e^2 (t - 1)^2 t^5.
\end{align}
Here, $e$ and $f$ are two complex parameters.
The Weierstrass equation (\ref{EK12}) defines the elliptic fibration $(x,y,t) \mapsto t$.
For a generic point $(e,f)$, 
the equation (\ref{EK12}) gives an elliptic surface with singular fibres $I_2^*, I_3$ and $II^*$ at $t=0,1$ and $\infty $,  respectively.

\begin{prop}
The $K3$ surface given by (\ref{EK12}) is birationally equivalent to the elliptic $K3$ surface given by the Weierstrass equation
\begin{align}\label{K3.12}
y_1^2=& x_1^3 +\Big(\frac{1}{3} (-9 e + 15 e f - f^4) u^4+   (-e^3 (1 + f)) u^5\Big)x \notag\\ 
&+\Big( u^5 + \frac{1}{27} (-54 e^2 - 81 e f^2 + 63 e f^3 + 2 f^6) u^6 + \frac{1}{3} e^3 (3 + 3 e + 3 f - 2 f^2 - 2 f^3)u^7 \Big),
\end{align}
with singular fibres of type $II^* + 5 I_1 + III^*$.
\end{prop}

\begin{proof}
Putting
\begin{align}\label{12:2neighbor}
x = ( t- 1) t^3 \Big(\frac{e}{( f-1) t} + u_1\Big), \quad y =  t^3 (t-1) y_0 
\end{align}
to (\ref{EK12}), we obtain 
an equation in the form
\begin{align}\label{4jishiki}
y_0^2 = \text{a polynomial in }  t \text{ of degree} 4.
 \end{align}
 Applying  the canonical method of Section 3.1
 in \cite{AKMMMP} to (\ref{4jishiki}),
we have an equation in the form
\begin{align}\label{preK31}
y_0^2=4 x_0^3 -I_0 (u_1) x_0 -J_0(u_1).
\end{align}
By the birational transformation
$
x_0 = \frac{e^6 x_1}{4 (-1 + f)^2},  u_1 = -\frac{e^3 u}{-1 + f},   y_0 = \frac{e^9 y_1}{4 (-1 + f)^3}
$
to (\ref{preK31}),
we have (\ref{K3.12}).
\end{proof}

\begin{rem}
The birational transformation (\ref{12:2neighbor}) gives an example of $2$-neighbor step
to obtain the singular fibre of type $III^*$. See Figure 3.
\end{rem}

\begin{figure}[h]
\center
\includegraphics[scale=0.6]{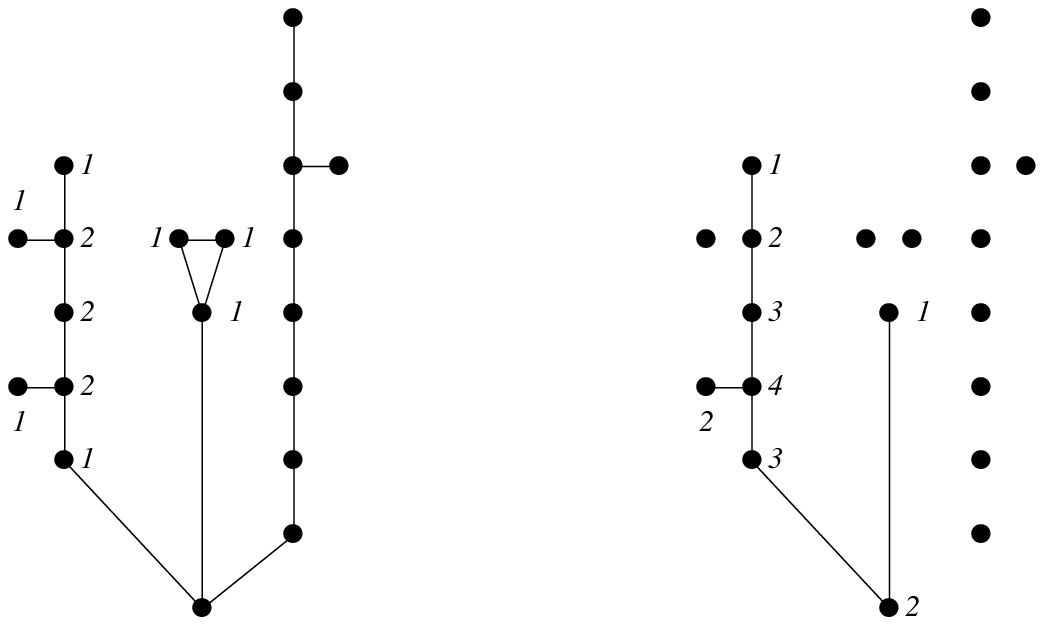}
\caption{2-neiborgh step in (\ref{12:2neighbor})}
\end{figure}

From (\ref{Mother}) and (\ref{K3.12}),
we obtain the following proposition.

\begin{prop}
The mapping $\chi_{12} : \mathbb{P}(2:1:1) \rightarrow \mathbb{P}(2:3:5:6)= \mathbb{P}(4:6:10:12)$ given by 
$(e:f:g) \mapsto (\alpha_{12}(e:f:g):\beta_{12}(e:f:g):\gamma_{12}(e:f:g):\delta_{12}(e:f:g))$ where
\begin{align}\label{parameter12}
\begin{cases}
&\displaystyle  \alpha_{12}(e:f:g)=\frac{1}{9} (f^4 - 15 e f g + 9 e g^2),
\displaystyle  \beta_{12}(e:f:g)=\frac{1}{54} f (-2 f^6 - 63 e f^3 g + 54 e^2 g^2 + 81 e f^2 g^2),\\
&\vspace{2mm} \gamma_{12}(e:f:g)= e^3 g^3 (f + g) ,
 \delta_{12}(e:f:g)=\frac{1}{3} e^3 g^3 (-2 f^3 + 3 e g - 2 f^2 g + 3 f g^2 + 3 g^3),
\end{cases}
\end{align}
gives  a parametrization of the Humbert surface $\mathcal{H}_{12}$.
\end{prop}

\subsection{The case of discriminant $21$}

In \cite{EK} Section 11, they studied the elliptic $K3$ surface given by the Weierstrass equation
\begin{align}\label{EK21}
y^2=x^3 + (a_0 + a_1 t + a_2 t^2) x^2 + 2 t^2 (t - 1) (b_0 + b_1 t) x + 
 t^4 (t - 1)^2 (c_0 + c_1 t),
\end{align}
where
\begin{align}
\begin{cases}
& a_0=1, \quad a_1= -r^2 + 2 r s - 1, \quad a_2=(r - s)^2,\\
& b_0 = (r^2 - 1) (s - r)^2, \quad b_1 =(r^2 - 1) (s - r)^2 (r s - 1), \\
& c_0= (r^2 - 1)^2 (s - r)^4, \quad c_1= (r^2 - 1)^3 (s - r)^4.
\end{cases}
\end{align}
Here, $r $ and $s$ are two complex parameters.
For generic $(r,s)$, an elliptic surface given by (\ref{EK21}) has singular fibers $I_7,I_3$ and $II^*$ at $t=0,1$ and $\infty$ respectively.

\begin{prop}
The $K3$ surface given by (\ref{EK21}) is birationally equivalent to the elliptic $K3$ surface given by the Weierstrass equation
\begin{align}\label{EK21s}
y_1^2=& x_1^3+\Big(\frac{1}{3} (-9 - 30 r - 7 r^2 + 30 r^3 + 15 r^4 + 30 s + 26 r s - 30 r^2 s - 
  24 r^3 s - s^2) u^4 \notag\\
  &\hspace{8cm} -(-1 + r)^6 (1 + r)^4 (r - s) u^5 \Big) x_1 \notag\\
& \quad +u^5- \frac{2}{27}(r - s) (189 + 252 r - 280 r^2 - 441 r^3 + 63 r^4 + 189 r^5 \notag \\
& \quad \quad \quad \quad  \quad \quad \quad   \quad \quad \quad \quad \quad 
  +  27 r^6 - 63 s - 70 r s + 63 r^2 s + 72 r^3 s - s^2)u^6 \notag\\
   &\quad + \frac{1}{3} (-1 + r)^6 (1 + r)^4 (3 - 5 r^2 + 3 r^4 + 6 s + 10 r s - 
   6 r^3 s - 2 s^2 + 3 r^2 s^2)u^7.
\end{align}
with singular fibres of type $II^* + 5 I_1  + III^*$.
\end{prop}

\begin{proof}
By the birational transformation given by
\begin{align}\label{3neighbor}
\begin{cases}
 x=& (-1 + r^2) (r - s)^2 (-1 + t) t^2 x_0,\\
 y=& \frac{1}{2} (-1 + r^2) (r - s)^2 (-1 + t) t^2\\
 & \times 
 (2 t^3 u_0 + 2 (1 + x_0) 
 -  t (1 + r^2 (-1 + x_0) + x_0 - 2 r s x_0) \\
&\quad \quad -  t^2 (1 + 2 u_0 - r^2 (-1 + x_0) + x_0 - 2 s x_0 + 2 r (-1 + s x_0))),
\end{cases}
\end{align}
the equation (\ref{EK21}) is changed to an equation in  the form
\begin{align}\label{3DegDiv}
a_0+a_1 x_0 + a_2 t + a_3 x_0^2 + a_4 x_0 t + a_5 t^2 + a_6 x^3 + a_7 x^2 t + a_8 x t^2 + a_9 t^3 =0.
\end{align}
Using the method of  Section 3.2 in \cite{AKMMMP},  the equation  (\ref{3DegDiv}) is transformed to  the Weierstrass equation  in the form
\begin{align}\label{PreK32}
y_0^2= 4x_0^3 - 108 S(u_0) x_0 - 27 T(u_0).
\end{align}
Putting 
$
 x_0= \frac{1}{4} (-1 + r)^{10} (1 + r)^6 (r - s)^3 x_1,
y_0 =  \frac{1}{4} (-1 + r)^{15} (1 + r)^9 \sqrt{(r - s)^9} y_1,
 u_0=\frac{1}{2} (-1 + r)^5 (1 + r)^3 (r - s) u_1
$
to (\ref{PreK32}), we have (\ref{EK21s}).
\end{proof}

\begin{rem}
The birational transformation (\ref{3neighbor}) gives an example of $3$-neighbor step
to obtain the singular fibre of type $III^*$. See Figure 4.
\end{rem}

\begin{figure}[h]
\center
\includegraphics[scale=0.6]{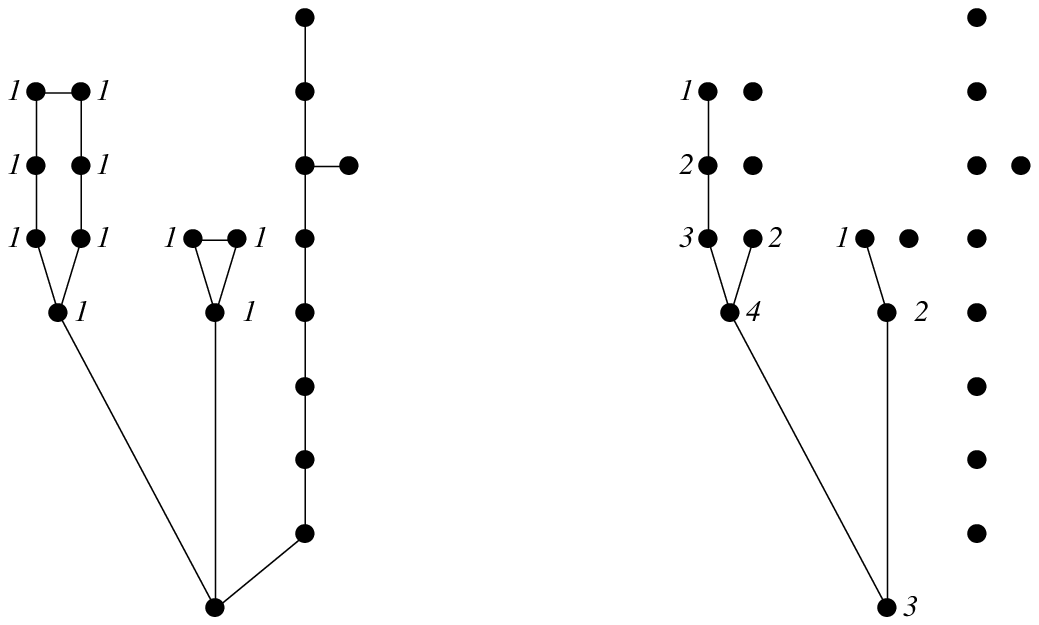}
\caption{3-neiborgh step for (\ref{3neighbor})}
\end{figure}

Setting
$
s = 1 + r_1 + \frac{r_1 (2 + r_1)}{s_1},  r = 1 + r_1
$
and comparing the coefficients of (\ref{Mother}) and (\ref{EK21s}), we can prove the following proposition.

\begin{prop}
The mapping $\chi_{21} : \mathbb{P}(1:1:2) \rightarrow \mathbb{P}(2:3:5:6) = \mathbb{P}(4:6:10:12) $ given by 
$(q_1:r_1:s_1) \mapsto (\alpha_{21}(q_1:r_1:s_1):\beta_{21}(q_1:r_1:s_1):\gamma_{21}(q_1:r_1:s_1):\delta_{21}(q_1:r_1:s_1))$ where
\begin{align}\label{ABCD21}
\begin{cases}
&\vspace{2mm}\displaystyle  \alpha_{21} (q_1:r_1:s_1)=\frac{1}{9} (q_1^4 + 54 q_1^2 s_1 + 24 q_1 r_1 s_1 + 9 s_1^2),\\
&\vspace{2mm}\displaystyle  \beta_{21} (q_1:r_1:s_1)=\frac{1}{27} (q_1^6 - 135 q_1^4 s_1 - 72 q_1^3 r_1 s_1 - 405 q_1^2 s_1^2 - 243 q_1 r_1 s_1^2 - 
  27 r_1^2 s_1^2),\\
&\vspace{2mm} \gamma_{21} (q_1:r_1:s_1)= -r_1^2 s_1^4 ,
\quad \delta_{21} (q_1:r_1:s_1)=\frac{1}{3} s_1^4 (q_1^2 r_1^2 + 6 q_1 r_1^3 + 3 r_1^4 + 6 q_1 r_1 s_1 + 3 s_1^2).
\end{cases}
\end{align}
gives a parametrization of the Humbert surface $\mathcal{H}_{21}$.
\end{prop}

\section{The Shimura curves of discriminant $6$  and $10 $ in $\mathbb{P}(1:3:5)$}

\subsection{The Shimura curves $\Psi^*_5(\mathcal{S}_6)$ and $\Psi^*_5(\mathcal{S}_{10})$}

Let us recall that
the Humbert surface $\mathcal{H}_\Delta$ of discriminant $\Delta$ is a surface in the Igusa 3-fold
$\mathcal{A}_2$.
Moreover, $\mathcal{A}_2$ is a Zariski open set of the weighted projective space $\mathbb{P}(2:3:5:6)={\rm Proj}(\mathbb{C}[\alpha:\beta:\gamma:\delta ]).$
Hence, we can regard the Humbert surface $\mathcal{H}_\Delta$ as a divisor in $\mathbb{P}(2:3:5:6)$.
Especially, Theorem \ref{PsiThm} says that 
 the Humbert surface $\mathcal{H}_5 \subset \mathcal{A}_2$  is parametrized by $(\mathfrak{A}:\mathfrak{B}:\mathfrak{C}) \in \mathbb{P}(1:3:5) \simeq \overline{  \langle PSL(2,\mathcal{O}_5),\tau \rangle \backslash  (\mathbb{H}\times \mathbb{H}) }$
 via the embedding $\Psi_5:\mathbb{P}(1:3:5) \hookrightarrow \mathbb{P}(2:3:5:6)$.
 On the other hand,
 the Humbert surface $\mathcal{H}_8 \subset \mathcal{A}_2$ is parametrized by $\chi_8$ (see Section 3.2).
The intersection $\mathcal{H}_5 \cap \mathcal{H}_8$ of two Humbert surfaces is an analytic subset in $\mathcal{A}_2$.
Let us consider the pull-back $\Psi^{*}_5 (\mathcal{H}_5 \cap \mathcal{H}_8)$, that is a curve in $\mathbb{P}(1:3:5)$.

\begin{thm}
\label{ShimuraModular}
The divisor $\Psi_5^*(\mathcal{H}_5 \cap \mathcal{H}_8)$ in the weighted projective space $\mathbb{P}(1:3:5)={\rm Proj}(\mathbb{C}[\mathfrak{A}:\mathfrak{B}:\mathfrak{C}])$ is given by the defining equation
\begin{align}\label{MonoGene}
& (\mathfrak{A}^5 - 5 \mathfrak{A}^2 \mathfrak{B} + \mathfrak{C}) \notag \\
 \times & (3125 \mathfrak{A}^5 - 3375 \mathfrak{A}^2 \mathfrak{B} + 
   243 \mathfrak{C}) \notag \\
\times  &  (38400000000000 \mathfrak{A}^9 \mathfrak{B}^2 + 120528000000000 \mathfrak{A}^6 \mathfrak{B}^3\notag\\
& - 
   4100625000000 \mathfrak{A}^3 \mathfrak{B}^4 - 184528125000000 \mathfrak{B}^5 + 
   2560000000000 \mathfrak{A}^10 \mathfrak{C} \notag\\
   & + 6998400000000 \mathfrak{A}^7 \mathfrak{B} \mathfrak{C} + 
   34698942000000 \mathfrak{A}^4 \mathfrak{B}^2 \mathfrak{C} + 42539883750000 \mathfrak{A} \mathfrak{B}^3 \mathfrak{C} \notag\\
   & + 
   2576431800000 \mathfrak{A}^5 \mathfrak{C}^2 + 2714325066000 \mathfrak{A}^2 \mathfrak{B} \mathfrak{C}^2 + 
   146211169851 \mathfrak{C}^3)=0.
   \end{align}
\end{thm}

\begin{proof}
We have the parametrization $\Psi_5$ in (\ref{(ABCDXY)}) ($\chi_8$ in (\ref{parameter8}), resp.) of the Humbert surface $\mathcal{H}_5$ ($\mathcal{H}_8$, resp.).

For a generic $(\alpha:\beta:\gamma:\delta) \in \mathcal{H}_5 \cap \mathcal{H}_8 \subset \mathbb{P}(2:3:5:6)$,
there exist $(\mathfrak{A}:\mathfrak{B}:\mathfrak{C}) \in \mathbb{P}(1:3:5)$ and $(r:s:q)\in \mathbb{P}^2(\mathbb{C})$
such that
\begin{align}\label{ABCDAtarimae}
\begin{cases}
& \alpha=\alpha_5(\mathfrak{A}:\mathfrak{B}:\mathfrak{C})=\alpha_8(r,s,q),\\
& \beta=\beta_5(\mathfrak{A}:\mathfrak{B}:\mathfrak{C})=\beta_8(r,s,q),\\
& \gamma=\gamma_5(\mathfrak{A}:\mathfrak{B}:\mathfrak{C})=\gamma_8(r,s,q),\\
& \delta=\delta_5(\mathfrak{A}:\mathfrak{B}:\mathfrak{C})=\delta_8(r,s,q).\\
\end{cases}
\end{align}
We have the  polynomial $F_1^{(5,8)}$ ($F_2^{(5,8)},F_3^{(5,8)},F_4^{(5,8)}$, resp) in $\mathbb{C}[\mathfrak{A}:\mathfrak{B}:\mathfrak{C},r,s,q] $ of  weight $2$ ($3,5,6$, resp.) :
\begin{align}
\begin{cases}
& F_1^{(5,8)}(\mathfrak{A},\mathfrak{B},\mathfrak{C},r,s,q)=\alpha_5(\mathfrak{A}:\mathfrak{B}:\mathfrak{C})-\alpha_8(r,s,q),\\
& F_2^{(5,8)} (\mathfrak{A},\mathfrak{B},\mathfrak{C},r,s,q)=\beta_5(\mathfrak{A}:\mathfrak{B}:\mathfrak{C})-\beta_8(r,s,q),\\
&F_3^{(5,8)}(\mathfrak{A},\mathfrak{B},\mathfrak{C},r,s,q)=\gamma_5(\mathfrak{A}:\mathfrak{B}:\mathfrak{C})-\gamma_8(r,s,q),\\
&F_4^{(5,8)}(\mathfrak{A},\mathfrak{B},\mathfrak{C},r,s,q)=\delta_5(\mathfrak{A}:\mathfrak{B}:\mathfrak{C})-\delta_8(r,s,q).\\
\end{cases}
\end{align}

Then, we have the weighted homogeneous ideal
$$
I=\langle F_1^{(5,8)},F_2^{(5,8)},F_3^{(5,8)},F_4^{(5,8)} \rangle \subset \mathbb{C}[\mathfrak{A},\mathfrak{B},\mathfrak{C},r,s,q].
$$
Let $V(I)$ the zero set of the ideal $I$.
 This is an analytic subset of $\mathbb{P}(1:3:5:1:1:1)={\rm Proj}(\mathbb{C}[\mathfrak{A},\mathfrak{B},\mathfrak{C},r,s,q])$. 
From (\ref{ABCDAtarimae}), the point $(\mathfrak{A}:\mathfrak{B}:\mathfrak{C}:r:s:q) \in V(I)$ gives $(\alpha:\beta:\gamma:\delta)\in \mathcal{H}_5\cap \mathcal{H}_8$.
Let $v:\mathbb{P}(1:3:5:1:1:1) \rightarrow \mathbb{P}(1:3:5)$
 be the canonical projection
given by 
$(\mathfrak{A}:\mathfrak{B}:\mathfrak{C}:r:s:q) \mapsto (\mathfrak{A}:\mathfrak{B}:\mathfrak{C})$.
The Zariski closure  of the image $v(V(I))$ corresponds to the zero set $V(I_S)$ of the elimination ideal
$$
I_S=I \cap \mathbb{C}[\mathfrak{A},\mathfrak{B},\mathfrak{C}].
$$
From (\ref{ABCDAtarimae}) again,
it follows that $(\alpha_5(\mathfrak{A}:\mathfrak{B}:\mathfrak{C}):\beta_5(\mathfrak{A}:\mathfrak{B}:\mathfrak{C}):\gamma_5(\mathfrak{A}:\mathfrak{B}:\mathfrak{C}):\delta_5(\mathfrak{A}:\mathfrak{B}:\mathfrak{C})) \in \mathcal{H}_5 \cap \mathcal{H}_8$ 
if and only if $(\mathfrak{A}:\mathfrak{B}:\mathfrak{C}) \in V(I_S)$.
 
By  the theory of Gr\"obner basis (see \cite{CLO})
and a computer aided calculation powered by Mathematica, 
we can show that the ideal $I_S$ is a principal ideal generated by  the polynomial  in (\ref{MonoGene}).
Thus, the theorem is proved.
\end{proof}

Using the explicit  expression of the parameters of $\mathcal{F}$ in (\ref{Ourtheta}), 
let us study the divisor $\Psi_5^* (\mathcal{H}_5 \cap \mathcal{H}_8)  \subset \mathbb{P}(1:3:5)$ in detail.
According to Example \ref{Exap5,8}, 
$\mathcal{H}_5 \cap \mathcal{H}_8$ contains the Shimura curves $\mathcal{S}_6$ and $\mathcal{S}_{10}$
as irreducible components. 
We shall give explicit forms of the pull-backs of these two Shimura curves as divisors in $\mathbb{P}(1:3:5)$.
We note that the pull-back $\Psi_5^*(\mathcal{S}_6)$ ($\Psi_5^*(\mathcal{S}_{10})$, resp.) is isomorphic to $\mathcal{S}_6$ ($\mathcal{S}_{10}$, resp.)
as varieties because $\Psi_5$ is an embedding of varieties and $\mathcal{S}_6 $ and $\mathcal{S}_{10}$ is contained in the image ${\rm Im}(\Psi_5)$.

Set
\begin{align}\label{R1R2R3}
\begin{cases}
&R_1:\mathfrak{A}^5 - 5 \mathfrak{A}^2 \mathfrak{B} + \mathfrak{C}=0 ,\\
&R_2:3125 \mathfrak{A}^5 - 3375 \mathfrak{A}^2 \mathfrak{B} + 
   243 \mathfrak{C} =0,\\
 & L_1: 38400000000000 \mathfrak{A}^9 \mathfrak{B}^2 + 120528000000000 \mathfrak{A}^6 \mathfrak{B}^3\\
& \quad \quad   - 
   4100625000000 \mathfrak{A}^3 \mathfrak{B}^4 - 184528125000000 \mathfrak{B}^5 + 
   2560000000000 \mathfrak{A}^10 \mathfrak{C} \\
   & \quad \quad + 6998400000000 \mathfrak{A}^7 \mathfrak{B} \mathfrak{C} + 
   34698942000000 \mathfrak{A}^4 \mathfrak{B}^2 \mathfrak{C} + 42539883750000 \mathfrak{A} \mathfrak{B}^3 \mathfrak{C} \\
   & \quad \quad + 
   2576431800000 \mathfrak{A}^5 \mathfrak{C}^2 + 2714325066000 \mathfrak{A}^2 \mathfrak{B} \mathfrak{C}^2 + 
   146211169851 \mathfrak{C}^3=0.
\end{cases}
\end{align}

\begin{lem}\label{R3lem}
The curve $L_1$ in (\ref{R1R2R3}) is  neither $\Psi_5^*(\mathcal{S}_6)$ nor $\Psi_5^*(\mathcal{S}_{10}).$
\end{lem}

\begin{proof}
By a direct calculation, 
the divisor $L_1$ intersects the divisor $\{\mathfrak{C}=0\}$ at the three points 
\begin{align}\label{L1Diag}
(\mathfrak{A}:\mathfrak{B}:\mathfrak{C})=(1:0:0),\Big(1:\frac{25}{27}:0\Big), \Big(1:-\frac{64}{135}\Big).
\end{align}

On the other hand, the Shimura curves $\mathcal{S}_6, \mathcal{S}_{10}$ are already compact. Then, $\Psi_5^*(\mathcal{S}_6)$ and $\Psi_5^*(\mathcal{S}_{10})$
never touch the cusp of $ \langle PSL(2,\mathcal{O}_5), \tau \rangle \backslash (\mathbb{H}\times \mathbb{H})$.
According to Proposition \ref{KleinH} (2),
the cusp  is given by $(\mathfrak{A}:\mathfrak{B}:\mathfrak{C})=(1:0:0)$.
So, from (\ref{L1Diag}), the curve $L_1$ is neither $\Psi^*_5(\mathcal{S}_6)$ nor $\Psi^*_5(\mathcal{S}_{10})$.
\end{proof}

According to \cite{HM},
we have the quaternion modular embedding $\check{\Omega}_6:\mathbb{H} \rightarrow \mathfrak{S}_2$ for $D=6$ given by
\begin{align}\label{Omega6}
w \mapsto \check{\Omega}_6(w)=
\begin{pmatrix}
\vspace{2mm} \displaystyle \frac{3w}{2} - \frac{1}{4w} &  \displaystyle -\frac{3\sqrt{2} w}{4} -\frac{1}{2} - \frac{\sqrt{2}}{8w}\\
\displaystyle -\frac{3\sqrt{2} w}{4} -\frac{1}{2} - \frac{\sqrt{2}}{8w} & \displaystyle \frac{3w}{4} - \frac{1}{2} -\frac{1}{8w}
\end{pmatrix}
\end{align}
for $w\in \mathbb{H}$.

\begin{rem}
The modular embedding $\check{\Omega}_6$ is different from the embedding $\Omega$ in Section 1.3.
However,
as we noted in Remark \ref{RemRot},
we have the unique choice of the Shimura curve $\mathcal{S}_6$.
So, our argument for  discriminant $6$ is not dependent on the choice of a quaternion modular embeddings.
\end{rem}

For
$$
M=\begin{pmatrix}
-1 & 0 & 1 &0\\
0 &0 &0 &1 \\
 -1 &0 &0 &0 \\
0 &-1 &0 &0 \\
\end{pmatrix}
=
\begin{pmatrix}
A &  B \\
C & D
\end{pmatrix}
\in Sp(4,\mathbb{Z}).
$$
and $w \mapsto \check{\Omega}_6 (w)$ in (\ref{Omega6}),
we  have another modular embedding $ \mathbb{H} \rightarrow \mathfrak{S}_2$ given by $w \mapsto \tilde{\Omega}_6(w) $ where
\begin{align}
\tilde{\Omega}_6(w) & = (A \check{\Omega}_6 (w) + B) (C \check{\Omega}_6 (w) +D)^{-1}\notag \\
&\label{tilde6}
=
\frac{1}{-1+\sqrt{2} +8 w + 6(1+\sqrt{2})w^2}
\begin{pmatrix}
-2+\sqrt{2} +4 w +6(2+\sqrt{2})w^2 & \sqrt{2}+4w +6\sqrt{2} w^2 \\
\sqrt{2} + 4 w + 6 \sqrt{2} w^2 & -2 +12 w^2 
\end{pmatrix}.
\end{align}
Setting $\tilde{\Omega}_6(w)=\begin{pmatrix} \tilde{\tau}_1(w) & \tilde{\tau}_2(w) \\ \tilde{\tau}_2 (w) & \tilde{\tau}_3(w)  \end{pmatrix}$,
it holds
\begin{align}
-\tilde{\tau}_1 (w) + \tilde{\tau}_2 (w) + \tilde{\tau}_3 (w ) =0.
\end{align}
Especially, 
$\tilde{\Omega}_6$  embeds $\mathbb{H}$ to  $N_5$ in (\ref{N5}).
Recall that the surface $N_5 $ is parametrized by the Hilbert modular embedding $\mu_5$ in (\ref{Mullerembedding}).

\begin{lem}\label{lemR1}
Let $\square$ be the diagonal,
 $\mu_5$ be the Hilbert modular embedding  given by (\ref{Mullerembedding})
and
 $pr$ be the canonical projection $\mathfrak{S}_2\rightarrow  Sp(4,\mathbb{Z}) \backslash  \mathfrak{S}_2 =\mathcal{A}_2$.
Set $M_6=pr \circ \mu_5 (\square) (\subset \mathcal{A}_2)$.
Then, $\mathcal{S}_6$ intersects $M_6$ at only one point 
$pr \circ \mu_5 (1+\sqrt{-1},1+\sqrt{-1})$. 
\end{lem}

\begin{proof} 
The embedding $w \mapsto \tilde{\Omega}_6 (w)$ parametrizes a curve in the surface $N_5$.
The surface $N_5$ is parametrized by $(z_1,z_2)\in \mathbb{H}\times \mathbb{H}$
via $\mu_5$.
Hence,
the set $\mu_5(\square) \cap {\rm Im}(\tilde{\Omega}_6)$ is given by the condition $\tilde{\tau_2}(w)=6 \sqrt{2} w^2 +4 w +\sqrt{2} =0$. 
The solution in the upper half plane $\mathbb{H}$ of the equation $6 \sqrt{2} w^2 +4 w +\sqrt{2} =0 $ is only
$$
w=\frac{1}{3}\Big(\sqrt{-1}-\frac{1}{\sqrt{2}}\Big).
$$
By a direct calculation, we can check  that 
$$
\tilde{\Omega}_6 (\frac{1}{3}\Big(\sqrt{-1}-\frac{1}{\sqrt{2}}\Big)) = \mu_5 (1+\sqrt{-1},1+\sqrt{-1}).
$$
\end{proof}

We note that a point in $N_5={\rm Im}(\mu_5)$ 
has an expression by the period mapping for the family $\mathcal{F}$ of $M_5$-marked $K3$ surfaces  (see Section 2.3).
The explicit theta expression (\ref{Ourtheta}) of the inverse of the period mapping for $\mathcal{F}$
enables us to study  the quaternion embedding $\tilde{\Omega}_6$ given by (\ref{tilde6}) in detail.

\begin{thm}\label{ThmS6S10}
The Shimura curve $\Psi_5^*(\mathcal{S}_6)$ ($\Psi_5^*(\mathcal{S}_{10})$, resp.) is given by
the divisor $R_2$ ($R_1$, resp.) in (\ref{R1R2R3}).
\end{thm}

\begin{proof}
According to Theorem \ref{ShimuraModular},
the Shimura curves $\Psi_5^*(\mathcal{S}_6)$ and $\Psi_5^* (\mathcal{S}_{10})$
are irreducible components of the union of the curves $R_1\cup  R_2 \cup L_1$.
However, from Lemma \ref{R3lem},
 the  curve $L_1$ never give any Shimura curve.
According to (\ref{XY}),
the $(X,Y)$-plane gives an affine plane of $\mathbb{P}(1:3:5)$.
Due to Lemma \ref{lemR1},
the Shimura curve $\Psi_5^*(\mathcal{S}_6)$ passes the point
$$
P_0=(X(1+\sqrt{-1}, 1+\sqrt{-1}),Y(1+\sqrt{-1}, 1+\sqrt{-1})).
$$
Since we have the formula (\ref{XYDiag}), 
$$
P_0=\Big(\frac{25}{27} \frac{1}{J(1+\sqrt{-1})}, 0\Big)=\Big(\frac{25}{27},0\Big)\in (X,Y)-\text{plane}.
$$

On the other hand, by a direct observation,
the curve $R_1$ does not touch the point $P_0$ and the curve $R_2$ passes the point $P_0$.

Therefore,
the Shimura curve $\Psi_5^* (\mathcal{S}_6)$ is given by the curve $R_2$.
The other curve $R_1$ corresponds to the Shimura curve $\Psi_5^* (\mathcal{S}_{10})$.  
\end{proof}

In Figure 5,
the Shimura curves $R_2=\Psi_5^*(\mathcal{S}_6)$, $R_1=\Psi_5^*(\mathcal{S}_{10})$
and the curve coming from Klein's icosahedral equation
$$
-1728\mathfrak{B}^5 +720\mathfrak{A}\mathfrak{C}\mathfrak{B}^3 -80 \mathfrak{A}^2 \mathfrak{C}^2 \mathfrak{B} +64\mathfrak{A}^3(5\mathfrak{B}^2-\mathfrak{A}\mathfrak{C})^2+\mathfrak{C}^3=0
$$
(see (\ref{KleinRel})).

\begin{figure}[h]
\center
\includegraphics[scale=0.22,clip]{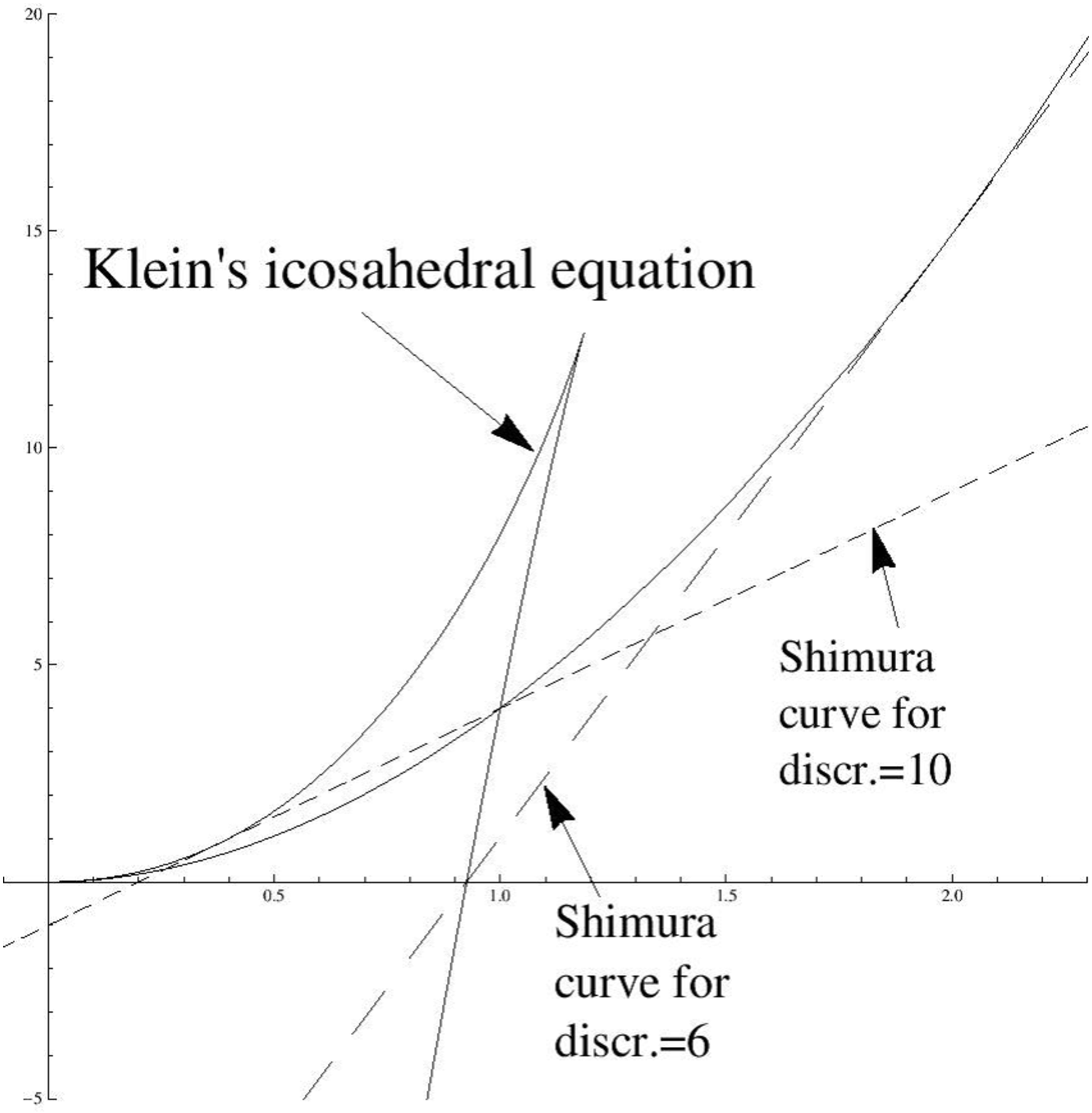}
\caption{The Shimura curves of discriminant $6$ and $10$ and the Klein's icosahedral equation}
\end{figure}

\subsection{The genus $2$ curves of Hashimoto and Murabayashi and the family $\mathcal{K}_j$  of Kummer surfaces }
Hashimoto and Murabayashi \cite{HM} 
studied  the moduli space of  genus $2$ curves  and the Shimura curves for discriminant $6$ and $10$.
In this subsection, let us see the relation between the results of \cite{HM} and Theorem \ref{ThmS6S10}.

Let $C$ be a Riemann surface of genus $2$.
The Jacobian variety ${\rm Jac}(C)$ is a principally polarized abelian surface.
Let $T$ be the involution on ${\rm Jac}(C)$ induced by $(z_1,z_2) \mapsto (-z_1, -z_2)$ on the universal covering $\mathbb{C}^2$.
The minimal resolution $\overline{{\rm Jac}(C)/\langle  id, T\rangle}$ is called the Kummer surface and denoted by ${\rm Kum}(C)$.

For a Riemann surface of  genus $2$  given by
 \begin{align} \label{genus2}
 C(\lambda_1,\lambda_2,\lambda_3) : Y^2 =X (X-1) (X-\lambda_1)(X-\lambda_2) (X-\lambda_3),
 \end{align}
 Humbert \cite{Humbert} obtained explicit conditions when the corresponding Jacobian variety ${\rm Jac}( C(\lambda_1,\lambda_2,\lambda_3))$ has 
 real multiplication for $\Delta=5$ and $\Delta=8$.
These conditions are given by equations, 
called Humbert's modular equations, 
in $\lambda_1,\lambda_2$  and $\lambda_3$ (see \cite{HM} Theorem 2.9 and 2.11).
For example,  Humbert's module equation for $\Delta=5$ is given by
\begin{align}\label{HumbertModular}
& 4(\lambda_1^2 \lambda_3 - \lambda_2^2 +\lambda_3^2 (1-\lambda_1) + \lambda_2 ^2\lambda_3)(\lambda_1^2\lambda_2\lambda_3 -\lambda_1 \lambda_2^2 \lambda_3) \notag \\
&=(\lambda_1^2(\lambda_2+1)\lambda_3 -\lambda_2^2 (\lambda_1 +\lambda_3) +(1-\lambda_1)\lambda_2\lambda_3^2 +\lambda_1(\lambda_2 -\lambda_3)  )^2.
\end{align}
It is well known that ${\rm Kum}(C(\lambda_1,\lambda_2,\lambda_3))$ is given by the  double cover of the projective plane
$\mathbb{P}^2(\mathbb{C})={\rm Proj} (\mathbb{C}[\zeta_0,\zeta_1,\zeta_2])$ branched along $6$ lines
$\zeta_2=0, \zeta_0=0, \zeta_2 + 2 \zeta_1 + \zeta_0=0$ and $\zeta_2 + 2 \lambda_j \zeta_1 + \lambda_j^2 \zeta_0=0$ $(j=1,2,3)$.
Humbert's modular equations for $\Delta=5$ and $\Delta=8$
are obtained by  a study of ${\rm Kum}(C(\lambda_1,\lambda_2,\lambda_3))$.

\begin{rem}\label{KummerRem}
Let
 $\mathcal{M}_{2,2}$ be the moduli space of genus $2$ curves with level $2$ structure
and $\mathcal{Q}:\mathcal{M}_{2,2} \rightarrow \mathcal{A}_2$ be the canonical projection.
Humbert's modular equation for $\Delta$ is not a defining equation of the Humbert surface $\mathcal{H}_\Delta \subset \mathcal{A}_2$,
but defines a component of $\mathcal{Q}^{-1}(\mathcal{H}_\Delta)$. 
To the best of the author's knowledge, 
to study Humbert's modular equations  is not easy, for they have complicated forms in $\lambda_1,\lambda_2$ and $\lambda_3$.
However, Humbert's modular equation (\ref{HumbertModular}) for $\Delta=5$ and our simple modular equation (\ref{modular5}) for $\mathcal{H}_5$ are explicitly related by the formula of \cite{NaganoShiga} Theorem 8.7.
\end{rem}

Hashimoto and Murabayashi \cite{HM} studied the genus $2$ curve $C(\lambda_1,\lambda_2,\lambda_3)$,
where $(\lambda_1,\lambda_2,\lambda_3)$ satisfies Humbert's modular equations $\Delta=5$ and $\Delta=8$.
Applying  Theorem \ref{HashimotoC} (2), they obtained the following results.

\begin{prop}  (\cite{HM}, Theorem 1.3, 1.7) \label{ThmHM}

(1) Set a genus $2$ curve $C_6(s,t)$ given by
$$
C_6(s,t): Y^2 = X (X^4 -P X^3 + Q X^2 -R X +1)
$$
with
$
 P=-2(s+t),
 Q=\displaystyle \frac{(1+2t^2) (11-28 t^2 +8t^4)}{3 (1-t^2)(1-4t^2)}$ and  $R=-2 (s-t)$.
Here,
$
(s,t) 
$
satisfies
\begin{align} \label{Cover6}
4 s^2 t^2 -s^2 + t^2 +2 =0.
\end{align}
Then, ${\rm Jac}(C_6(s,t))$ is a principally polarized abelian surface with  quaternion multiplication by $\mathfrak{O}_6$.

(2)  Set a genus $2$ curve $C_{10}(s,t)$ given by
$$
C_{10}(s,t): Y^2 = X (X^4 -P X^3 + Q X^2 -R X +1)
$$
with
$
 P=\displaystyle \frac{4(2 t +1 )( t^2 -t-1)}{ (t-1)^2 }, Q=\displaystyle \frac{(1+t^2) (t^4 + 8t^3 -10 t^2 -8 t +1)}{t (t-1)^2 (t+1)^2}$
and $ R=\displaystyle  \frac{(t-1) s}{ t (t+1) (2 t +1)}$.
Here,
$
(s,t) 
$
satisfies
\begin{align} \label{Cover10}
s^2  -  t (t-2)(2 t+1) =0.
\end{align}
Then, ${\rm Jac}(C_{10}(s,t))$ is a principally polarized abelian surface with  quaternion multiplication by $\mathfrak{O}_{10}$.
\end{prop}

\begin{rem}
The equation  (\ref{Cover6}) ((\ref{Cover10}), resp.) does not give the exact defining equation of the Shimura curve $\mathcal{S}_6$
($\mathcal{S}_{10}$, resp.) but defines a covering of $\mathcal{S}_6$
($\mathcal{S}_{10}$, resp.).
In fact, each curve defined by (\ref{Cover6}) and (\ref{Cover10}) is of genus $1$.
On the other hand, 
$\mathcal{S}_6$ and  $\mathcal{S}_{10}$ are genus $0$ curves.
\end{rem}

Let $j\in\{6,10\}$.
Set $\mathcal{C}_j=\{C_j(s,t)\}$ in Proposition \ref{ThmHM}.
For two members $C_j(s_1,t_1)$ and $C_j(s_2,t_2)$ of $\mathcal{C}_j$,
if ${\rm Jac}(C_j(s_1,t_1))$ and ${\rm Jac}(C_j(s_2,t_2))$ are    isomorphic as principally polarized abelian surfaces,
we call two members are equivalent.
Let $[C_j(s,t)]$ be the equivalence class of  $C_j(s,t) \in \mathcal{C}_j$.
Let $\tilde{\mathcal{C}}_j$ denote the equivalent class of $\mathcal{C}_j$.
We have the family  ${\rm Kum}(\tilde{\mathcal{C}}_j)=
\{{\rm Kum}( C_j(s,t)) | [C_j(s,t)]  \in \tilde{\mathcal{C}}_j\}
$
of Kummer surfaces.

On the other hand,
our $K3$ surface $S(\mathfrak{A}:\mathfrak{B}:\mathfrak{C}) $ in (\ref{S(ABC)}) has the Shioda-Inose structure.
Namely, there exists  an involution $\sigma$ on 
$S(\mathfrak{A}:\mathfrak{B}:\mathfrak{C}) $
such that 
the minimal resolution $K(\mathfrak{A}:\mathfrak{B}:\mathfrak{C})$ of $S(\mathfrak{A}:\mathfrak{B}:\mathfrak{C}) /\langle id ,\sigma \rangle$ is a Kummer surface.
The Kummer surface $K(\mathfrak{A}:\mathfrak{B}:\mathfrak{C})$ is given by the following equation (see \cite{NaganoKummer} Theorem 2.13),
\begin{align}\label{ourKummer}
v^2=(u^2-2t^5)(u-(5\mathfrak{A}t^2-10\mathfrak{B}t+\mathfrak{C})).
\end{align}

\begin{rem}\label{KummerPerRem}
The period mapping for the family $\mathcal{F}=\{S(\mathfrak{A}:\mathfrak{B}:\mathfrak{C}) \}$ coincides with that of  the family $\mathcal{K}=\{K(\mathfrak{A}:\mathfrak{B}:\mathfrak{C})\}$ of Kummer surfaces. 
(see \cite{NaganoKummer} Section 2.4).
\end{rem}

From Theorem \ref{ThmS6S10} and the defining equation (\ref{ourKummer}), 
 we have two families of Kummer surfaces $\mathcal{K}_j=\{K_j (\mathfrak{A}:\mathfrak{B})\}$ ($j\in \{6,10\}$) given by
 \begin{align}
 \begin{cases}
 & K_6(\mathfrak{A}:\mathfrak{B}) :  v^2=(u^2-2t^5)(u-(5\mathfrak{A}t^2-10\mathfrak{B}t-\frac{3125 \mathfrak{A}^5 - 3375 \mathfrak{A}^2 \mathfrak{B}}{243})), \\
 &K_{10} (\mathfrak{A}:\mathfrak{B})  : v^2=(u^2-2t^5)(u-(5\mathfrak{A}t^2-10\mathfrak{B}t-( \mathfrak{A}^5 - 5 \mathfrak{A}^2 \mathfrak{B})). \\
 \end{cases}
 \end{align}
 
 Considering the properties of the $K3$ surface $S(\mathfrak{A}:\mathfrak{B}:\mathfrak{C}) $,
  the above procedure of the families $\mathcal{K}_j$ for $j=6,10$,   
 Remark \ref{RemRot}, Remark \ref{KummerRem} and Remark \ref{KummerPerRem},
 we have the following proposition.

 \begin{prop}
 For $j\in \{6,10\}$, the family ${\rm Kum}(\tilde{\mathcal{C}}_j)$ coincides with the family  $\mathcal{K}_j$.
 \end{prop}

\section{The Shimura curves of discriminant $14$  and $15 $ in $\mathbb{P}(1:3:5)$}

In this section, we obtain the explicit forms
of the Shimura curves for discriminants $14$ and $15$
 in the weighted projective space ${\rm Proj}(\mathbb{C}[\mathfrak{A}:\mathfrak{B}:\mathfrak{C}])$.
 
However, as in Remark \ref{RemRot},
 the Shimura curve
 $\mathcal{S}_D= \varphi_D (\mathbb{S}_D) \subset \mathcal{A}_2$ 
 is not unique 
 for $D=14$
  because
  there exist two choices of $\varphi_D$. 
So, 
the image $\mathcal{S}_{14}$ of  the  Shimura curve depends on the triples $(p,a,b)$ in  the argument of Section 2.4.

 In this section, as in Example \ref{Exap5,12} and \ref{Exap5,21},
we only consider the Shimura curve $\mathcal{S}_{14}$  
 in $\mathcal{A}_2$ coming form the triple $(p,a,b)= (5,1,3)$.

\begin{thm}\label{Thm5,12}
The pull-back $\Psi_5^*(\mathcal{H}_5 \cap \mathcal{H}_{12})$ is given by the union of four devisors $R_2, R_3, R_4, L_2$,
where
$R_2$ is given by (\ref{R1R2R3}) and $R_3,R_4$ and $L_2$ are curves in $\mathbb{P}(1:3:5)={\rm Proj}(\mathbb{C}[ \mathfrak{A}, \mathfrak{B}, \mathfrak{C}])$ in the following:
\begin{align*}
&R_3:1048576 \mathfrak{A}^{10} - 30965760 \mathfrak{A}^7 \mathfrak{B} - 144633600 \mathfrak{A}^4 \mathfrak{B}^2 \notag \\
& \quad\quad- 157464000 \mathfrak{A} \mathfrak{B}^3 + 72721152 \mathfrak{A}^5 \mathfrak{C} - 27293760 \mathfrak{A}^2 \mathfrak{B} \mathfrak{C} - 
  59049 \mathfrak{C}^2=0,
\end{align*}
\begin{align*}
&R_4:30517578125 \mathfrak{A}^{15} + 911865234375 \mathfrak{A}^{12} \mathfrak{B} + 42529296875 \mathfrak{A}^9 \mathfrak{B}^2 - 
  97897974609375 \mathfrak{A}^6 \mathfrak{B}^3 \notag \\
  & \quad \quad + 424490000000000 \mathfrak{A}^3 \mathfrak{B}^4 - 
  345600000000000 \mathfrak{B}^5 + 2383486328125 \mathfrak{A}^10 \mathfrak{C} \notag \\
&\quad \quad + 
  32875975781250 \mathfrak{A}^7 \mathfrak{B} \mathfrak{C} - 147816767984375 \mathfrak{A}^4 \mathfrak{B}^2 \mathfrak{C} + 
  228155760000000 \mathfrak{A} \mathfrak{B}^3 \mathfrak{C} \notag \\
  &\quad \quad + 19189204671875 \mathfrak{A}^5 \mathfrak{C}^2 - 
  29675018141125 \mathfrak{A}^2 \mathfrak{B} \mathfrak{C}^2 + 344730881243 \mathfrak{C}^3=0,
\end{align*}
{\small \begin{align*}
& L_2:   -64000000000000000000 \mathfrak{A}^{12} \mathfrak{B}^6 + 370000000000000000000 \mathfrak{A}^9 \mathfrak{B}^7 + 
  815234375000000000000 \mathfrak{A}^6 \mathfrak{B}^8    \notag \\
 & \quad - 3902343750000000000000 \mathfrak{A}^3 \mathfrak{B}^9 - 7119140625000000000000 \mathfrak{B}^10 + 
  38400000000000000000 \mathfrak{A}^{13} \mathfrak{B}^4 \mathfrak{C} \notag \\
& \quad  -  223920000000000000000 \mathfrak{A}^{10} \mathfrak{B}^5 \mathfrak{C} - 
  1075967500000000000000 \mathfrak{A}^7 \mathfrak{B}^6 \mathfrak{C} + 
  3969323437500000000000 \mathfrak{A}^4 \mathfrak{B}^7 \mathfrak{C} \notag \\
&\quad  + 
  4702429687500000000000 \mathfrak{A} \mathfrak{B}^8 \mathfrak{C} - 
  7680000000000000000 \mathfrak{A}^{14} \mathfrak{B}^2 \mathfrak{C}^2 + 
  45577600000000000000 \mathfrak{A}^{11} \mathfrak{B}^3 \mathfrak{C}^2 \notag\\
  &\quad + 
  449730698000000000000 \mathfrak{A}^8 \mathfrak{B}^4 \mathfrak{C}^2 - 
  1463038602500000000000 \mathfrak{A}^5 \mathfrak{B}^5 \mathfrak{C}^2 - 
  1122863301562500000000 \mathfrak{A}^2 \mathfrak{B}^6 \mathfrak{C}^2 \notag\\
 &\quad + 
 512000000000000000 \mathfrak{A}^{15} \mathfrak{C}^3 - 3077760000000000000 \mathfrak{A}^{12} \mathfrak{B} \mathfrak{C}^3 - 
  77561010400000000000 \mathfrak{A}^9 \mathfrak{B}^2 \mathfrak{C}^3   \notag \\ 
 &\quad + 
  235959322740000000000 \mathfrak{A}^6 \mathfrak{B}^3 \mathfrak{C}^3 + 
  121351323118750000000 \mathfrak{A}^3 \mathfrak{B}^4 \mathfrak{C}^3 + 
  13523702118750000000 \mathfrak{B}^5 \mathfrak{C}^3  \notag \\
& \quad   + 4779900760000000000 \mathfrak{A}^{10} \mathfrak{C}^4 - 
  13908191752800000000 \mathfrak{A}^7 \mathfrak{B} \mathfrak{C}^4 - 
  8326918293212000000 \mathfrak{A}^4 \mathfrak{B}^2 \mathfrak{C}^4 \notag \\ 
&\quad  -  2530877087227500000 \mathfrak{A} \mathfrak{B}^3 \mathfrak{C}^4 + 449415539646800000 \mathfrak{A}^5 \mathfrak{C}^5 \notag \\
& \quad + 
  103922033314060000 \mathfrak{A}^2 \mathfrak{B} \mathfrak{C}^5 + 50787635527751 \mathfrak{C}^6=0.
\end{align*}}
\end{thm}

\begin{proof}
Recalling (\ref{(ABCDXY)}) and (\ref{parameter12}),  set
\begin{align*}
\begin{cases}
& F_1^{(5,12)}(\mathfrak{A},\mathfrak{B},\mathfrak{C},e,f,g)=\alpha_5(\mathfrak{A},\mathfrak{B},\mathfrak{C})-\alpha_{12}(e,f,g),\\
& F_2^{(5,12)} (\mathfrak{A},\mathfrak{B},\mathfrak{C},e,f,g)=\beta_5(\mathfrak{A},\mathfrak{B},\mathfrak{C})-\beta_{12}(e,f,g),\\
&F_3^{(5,12)}(\mathfrak{A},\mathfrak{B},\mathfrak{C},e,f,g)=\gamma_5(\mathfrak{A},\mathfrak{B},\mathfrak{C})-\gamma_{12}(e,f,g),\\
&F_4^{(5,12)}(\mathfrak{A},\mathfrak{B},\mathfrak{C},e,f,g)=\delta_5(\mathfrak{A},\mathfrak{B},\mathfrak{C})-\delta_{12}(e,f,g)\\
\end{cases}
\end{align*}
and take the weighted homogeneous  ideal
$
I_{12}=\langle F_1^{(5,12)},F_2^{(5,12)}, F_3^{(5,12)},F_4^{(5,12)}  \rangle 
$
of the ring $\mathbb{C}[\mathfrak{A},\mathfrak{B},\mathfrak{C},e,f,g]$.
As in the proof of Theorem \ref{ShimuraModular},
the zero set of the ideal $I_{12,S}=I_{12}\cap \mathbb{C}[\mathfrak{A},\mathfrak{B},\mathfrak{C}]$
corresponds to the pull-back $\Psi_5^* (\mathcal{H}_5 \cap \mathcal{H}_{12})$.
By a computer aided calculation, we can show that the zero set of $I_{12,S}$ is the union of the curves $R_2,R_3,R_4$ and $L_2$.
\end{proof}

\begin{thm}\label{ThmS14S15}
The Shimura curve $\Psi_5^*(\mathcal{S}_{15})$  corresponds to $R_3$ and  
the Shimura curve $\Psi_{5}^*(\mathcal{S}_{14})$  corresponds to $R_4$.
\end{thm}

\begin{proof}
First, according to Theorem \ref{Thm5,12},  
the divisor $\Psi_5^* (\mathcal{H}_5 \cap \mathcal{H}_{12})$ 
consists of only four irreducible components $R_2, R_3, R_4$ and $L_2$.
However, from Theorem \ref{ThmS6S10},
 the curve $R_2 $ is  the Shimura curve  $\Psi_5^*(\mathcal{S}_6)$. 
 Moreover, since the curve $L_2$ passes the cusp  $(\mathfrak{A}:\mathfrak{B}:\mathfrak{C})=(1:0:0)$, 
 the curve $L_2$ does not corresponds to any Shimura curves.

 Then, we shall identify  the curves $R_3$ and $R_4$.
 Because we have Example \ref{Exap5,12}, 
 the divisor $\Psi_5^*(\mathcal{H}_5\cap \mathcal{H}_{12})$ contains the Shimura curves $\Psi_5^*(\mathcal{S}_{15})$ and $\Psi_5^*(\mathcal{S}_{14})$.
 So, 
 of the two curves $R_3$ and $R_4$,
 one corresponds to 
  $\Psi_5^*(\mathcal{S}_{15})$ and the other  corresponds to $\Psi_5^*(\mathcal{S}_{14})$.
 By the way,
 according to Example \ref{Exap5,21}, only the Shimura curve $\Psi_5^* (\mathcal{S}_{14})$ is contained in the divisor $\Psi_5^*(\mathcal{H}_5\cap \mathcal{H}_{21})$.  
 Moreover, due to the next lemma, 
the curve $R_4$ is an irreducible component  of  $\Psi_5^*(\mathcal{H}_5 \cap \mathcal{H}_{21})$.
Therefore, we conclude that the curve $R_3$ ($R_4$, resp.) gives the explicit model of the Shimura curve $\Psi_5^*(\mathcal{S}_{15})$, ($\Psi_5^*(\mathcal{S}_{14})$, resp.). 
\end{proof}

\begin{lem}
The curve $R_4$ is a irreducible component of the divisor $\Psi_5^* (\mathcal{H}_5 \cap \mathcal{H}_{21})$.
\end{lem}

\begin{proof}
Using the notation in (\ref{(ABCDXY)}) and (\ref{ABCD21}),
we set
\begin{align*}
\begin{cases}
&F_1^{(5,21)}(\mathfrak{A},\mathfrak{B},\mathfrak{C},q_1,r_1,s_1)= \alpha_5(\mathfrak{A},\mathfrak{B},\mathfrak{C}) - \alpha_{21}(q_1,r_1,s_1),\\
&F_2^{(5,21)}(\mathfrak{A},\mathfrak{B},\mathfrak{C},q_1,r_1,s_1)= \beta_5(\mathfrak{A},\mathfrak{B},\mathfrak{C}) - \beta_{21}(q_1,r_1,s_1),\\
&F_3^{(5,21)}(\mathfrak{A},\mathfrak{B},\mathfrak{C},q_1,r_1,s_1)=  \gamma_5(\mathfrak{A},\mathfrak{B},\mathfrak{C}) - \gamma_{21}(q_1,r_1,s_1),\\
&F_4^{(5,21)}(\mathfrak{A},\mathfrak{B},\mathfrak{C},q_1,r_1,s_1)=\delta_5(\mathfrak{A},\mathfrak{B},\mathfrak{C}) - \delta_{21}(q_1,r_1,s_1).
\end{cases}
\end{align*}
We take the weighted homogeneous  ideal
$$
I_{21}=\langle  F_1^{(5,21)},F_2^{(5,21)},F_3^{(5,21)},F_4^{(5,21)}\rangle
$$
in the ring $\mathbb{C}[\mathfrak{A},\mathfrak{B},\mathfrak{C},q_1,r_1,s_1]$.
The zero set  of the ideal 
$$
I_{21,S}=I_{21} \cap \mathbb{C}[\mathfrak{A},\mathfrak{B},\mathfrak{C}]
$$ 
gives the curve $\Psi_5^*( \mathcal{H}_5 \cap \mathcal{H}_{21})$.
 However, 
 because of a huge amount of calculations of the Gr\"obner basis,
 it is very difficult to obtain a system of generators of the ideal $I_{21,S}$ directly. 

Instead, 
we shall consider  the ideal
$$
  I_{21,T}=I_{21} \cap \mathbb{C}[q_1,r_1,s_1].
$$
By a computer aided calculation, we can show that the ideal $I_{21,T}$ is a principal ideal generated by
\begin{align}\label{genersq}
&q_1 (2 q_1 + r_1) s_1^2 (q_1^3 - q_1 s_1 - r_1 s_1) \notag\\
&\times (2 q_1^4 + q_1^3 r_1 - 11 q_1^2 s_1 - 
   22 q_1 r_1 s_1 - 8 r_1^2 s_1 + 9 s_1^2) \times (2 q_1^4 - 27 q_1^2 s_1 - 27 q_1 r_1 s_1 + 
   81 s_1^2) \notag\\
  & \times (q_1^6 - 20 q_1^4 s_1 - 9 q_1^3 r_1 s_1 + 98 q_1^2 s_1^2 + 62 q_1 r_1 s_1^2 + 
   8 r_1^2 s_1^2 + 9 s_1^3).
\end{align}

Next, 
taking a factor $q_1^3 - q_1 s_1 - r_1 s_1$ of (\ref{genersq}),
we consider the new ideal 
$$
J_{21}=\langle  F_1^{(5,21)},F_2^{(5,21)},F_3^{(5,21)},F_4^{(5,21)},  q_1^3 - q_1 s_1 - r_1 s_1 \rangle.
$$
The zero set $V(J_{21,S})$ of the elimination ideal 
$$
J_{21,S}=  J_{21} \cap \mathbb{C}[\mathfrak{A},\mathfrak{B},\mathfrak{C}]
$$
corresponds to certain components of $\Psi_5^*(\mathcal{H}_5 \cap \mathcal{H}_{21})$.
By a computer aided calculation, 
we can compute the Gr\"obner basis of the ideal $J_{21,S}$.
So, we can check that the curve $R_4$ is contained in the zero set of the ideal $J_{21,S}$.
\end{proof}

The author believes that 
the method using icosahedral invariants and the theta expression (\ref{Ourtheta})
is effective for the Shimura curves $\mathcal{S}_D$ for $D>15$
with some skillful computations of Gr\"obner basis.

\section*{Acknowledgment}
The author would like to thank Professor Hironori Shiga for helpful advises and  valuable suggestions,
and also to Professor Kimio Ueno  for kind encouragements.
He is grateful to the referee for careful comments. 
This work is supported by 
The JSPS Program for Advancing Strategic International Networks to Accelerate the Circulation of Talented Researchers
"Mathematical Science of Symmetry, Topology and Moduli, Evolution of International Research Network based on OCAMI",
The Sumitomo Foundation Grant for Basic Science Research Project (No.150108) 
and
Waseda University Grant for Special Research Project (2014B-169 and 2015B-191).

\begin{center}
\hspace{7.7cm}\textit{Atsuhira  Nagano}\\
\hspace{7cm}\textit{Department of Mathematics}\\
\hspace{7.7cm}\textit{King's College London}\\
\hspace{7.7cm}\textit{Strand, London, WC2R 2LS}\\
\hspace{7.7cm}\textit{United Kingdom}\\
 \hspace{7.7cm}\textit{(E-mail: atsuhira.nagano@gmail.com)}
  \end{center}


\begin{thebibliography}{20}


\bibitem[AKMMMP]{AKMMMP} S. Y. An,  S. Y. Kim,  D. C. Marshall, S.H. Marshall, W. G. McCallum and A. R. Perlis,
\newblock{\em Jacobians of Genus One Curves,}
\newblock{J. of Number Theory, {\bf 90} (2), 304-315.}


\bibitem[B]{Besser} A. Besser,
\newblock{\em Elliptic fibrations of $K3$ surfaces and QM Kummer surfaces,}
\newblock{Math. Z., {\bf 228} (2),  1998, 283-308.}


\bibitem[BG]{BG} M. A. Bonfanti and B. van Geemen,
\newblock{\em Abelian surfaces with an automorphism and quaternionic multiplication,}
\newblock{Canadian J. Math.,  2015, to appear. }


\bibitem[CD]{CD} A. Clingher and C. Doran,
\newblock{\em Lattice polarized $K3$ surfaces and Siegel modular forms,}
\newblock{Adv. Math., {\bf 231}, 2012, 172-212.}



\bibitem[CLO]{CLO} D. Cox, J. Little and D. O'Shea,
\newblock{\em Using algebraic geometry,}
\newblock{Springer, 1998.}


\bibitem[D]{Dolgachev} I. V. Dolgechev,
\newblock{\em Mirror symmetry for lattice polarized $K3$ surfaces,}
\newblock{J. Math. Sci., {\bf 81} (3),  1996, 2599-2630.}


\bibitem[E1]{Elkies1} N. Elkies,
\newblock{\em Shimura curve computations}
\newblock{Algorithmic Number Theory, Lect. Notes in Computer Sci., {\bf 1423}, Springer, 1998, 1-47. }



\bibitem[E2]{Elkies2} N. Elkies,
\newblock{\em Shimura curve computations via $K3$ surfaces of N\'eron-Severi rank at least $19$, }
\newblock{Proc. of the 8th conference on Algorithmic Number Theory, Springer, 2008, 196-211.}



\bibitem[EK]{EK} N. Elkies and A. Kumar,
\newblock{\em $K3$ Surfaces and equations for Hilbert modular surfaces, }
\newblock{arXiv:1209.3527v2.}



\bibitem[G]{Geer}
 G. van der Geer,
\newblock{\em Hilbert modular surfaces},
\newblock{Ergebnisse der Math. und ihrer Grenzgebiete 3-Folge {\bf 16}, Springer,  1988.}






\bibitem[HM]{HM}
 K. Hashimoto and N. Murabayashi,
\newblock{\em Shimura curves as intersections of  Humbert surfaces and defining equations  of  QM-curves of genus two},
\newblock{Tohoku  Math. J., {\bf 47},  1995, 271-296.}


\bibitem[HNU]{HNU}
K. Hashimoto, A. Nagano and K. Ueda, 
\newblock{\em  Modular surfaces associated with toric $K3$ surfaces}, 
\newblock{arXiv:1403.5818, prerprint, 2014. }


\bibitem[Ha]{Hashimoto}
 K. Hashimoto,
\newblock{\em Explicit forms of quaternion modular embeddings},
\newblock{Osaka   J. Math., {\bf 32}, 1995, 533-546.}


\bibitem[Hi]{Hirzebruch}F. Hirzebruch,
\newblock {\em The ring of Hilbert modular forms for real quadratic fields of small discriminant,}
\newblock {Lecture Notes in Math. {\bf 627}, Springer-Verlag, 1977, 287-323.}


 \bibitem[Hu]{Humbert} G. Humbert,
\newblock{\em Sur les fonctions ab\'{e}liennes singuli\`{e}res,}
\newblock{Oeuvres de G. Humbert 2, pub. par les soins de Pierre Humbert et de Gaston Julia, Gauthier-Villars, 297-401, 1936.}


\bibitem[Kl]{Klein}F. Klein,
\newblock{\em Vorlesungen \"{u}ber das Ikosaeder und die Aufl\"{o}sung der Gleichungen vom f\"{u}nften Grade,}
\newblock{Tauber, 1884.}


\bibitem[Kum]{Kumar}
A. Kumar,
\newblock{\em $K3$ surfaces associated to curves of genus two,}
\newblock{Int. Math. Re. Not., {\bf 16}, 2008, ArticleID: rnm165.}


\bibitem[Kur]{Kurihara}
A. Kurihara,
\newblock{\em On some examples of equations defining Shimura curves and the Mumford uniformization,}
\newblock{J. Fac. Sci. Univ. Tokyo {\bf 25}, 1979, 277-301.}

\bibitem[KV]{KohelVerrill}
D. Kohel and H. Verrill,
\newblock{\em Fundamental Domains for Shimura Curves,}
\newblock{J. Th\'eor. Nombres Bordeaux, {\bf 15}, 2003.}




\bibitem[M]{Muller}
R. M\"uller
\newblock{\em Hilbertsche Modulformen und Modulfunctionen zu $\mathbb{Q}(\sqrt{5})$,}
\newblock{Arch. Math. {\bf 45}, 1985, 239-251}


\bibitem[N1]{NaganoPDE} 
A. Nagano,
\newblock {\em Period differential equations for the families of $K3$ surfaces with two parameters derived from the reflexive polytopes,}
\newblock {Kyushu J. Math., {\bf 66} (1), 2012, 193-244.}


\bibitem[N2]{NaganoTheta} 
A. Nagano,
\newblock {\em A theta expression of the Hilbert modular functions for $\sqrt{5}$ via period of $K3$ surfaces,}
\newblock {Kyoto J. Math., {\bf 53} (4), 2013, 815-843.}

\bibitem[N3]{NaganoKummer} 
A. Nagano,
\newblock {\em Double integrals on a weighted projective plane and the Hilbert modular functions for $\mathbb{Q}(\sqrt{5})$,}
\newblock {Acta Arith., {\bf 167} (4), 2015, 327-345.}

\bibitem[N4]{NaganoCM}
A. Nagano,
\newblock {\em Icosahedral invariants and CM points and class fields,}
\newblock {preprint, 2015, arXiv:1504.07500.}

\bibitem[NS]{NaganoShiga} 
A. Nagano and H. Shiga,
\newblock {\em Modular map for the family of abelian surfaces via elliptic $K3$ surfaces,}
\newblock {Math. Nachr., {\bf 288} (1), 89-114, 2015.}



\bibitem[R1]{Rotger}
V. Rotger,
\newblock{\em Shimura curves embedded in Igusa's threefold,}
\newblock{Modular curves and abelian varieties (Progress in Math. {\bf 224}), Birkh\"auser,  2004, 263-276.}


\bibitem[R2]{Rotger2}
V. Rotger,
\newblock{\em Modular Shimura varieties and forgetful maps,}
\newblock{Trans. Amer. Math. Soc.  {\bf 356},  2004, 1535-1550.}


\bibitem[S1]{S67} 
G. Shimura,
\newblock {\em Construction of class fields and zeta functions of algebraic curves,}
\newblock {Ann. of Math., {\bf 85}, 1967, 58-159.}


\bibitem[S2]{S70}G. Shimura,
\newblock {\em On canonical models of arithmetic quotients of bounded symmetric domains I,}
\newblock {Ann. of Math. {\bf 91}, 1970, 144-222.}



\bibitem[S3]{Shimura} 
G. Shimura,
\newblock {\em Introduction to the arithmetic theory of automorphic functions,}
\newblock {Publ. Math. Soc. Japan {\bf 11}, 1971.}



\bibitem[S4]{S75}G. Shimura,
\newblock {\em On the real points of an arithmetic quotient of a bounded symmetric domain,}
\newblock {Math. Ann. {\bf 215}, 1975, 135-164.}

%\bibitem[S3]{S77}G. Shimura,
%\newblock {\em On abelian varieties with complex multiplication,}
%\newblock {Proc. London. Math. Soc. {\bf 34} (3),  1977, 65-86.}


\bibitem[S5]{S97}G. Shimura,
\newblock {\em Abelian Varieties with Complex Multiplication and Modular Functions,}
\newblock {Princeton Univ. Press, 1997.}


\bibitem[SW]{SW} 
H. P. F. Swinnerton-Dyer,
\newblock {\em Analytic Theory of Abelian Varieties,}
\newblock {London Mathematical Societiy Lecture Note Series {\bf 14}, 1974.}


\bibitem[Vi]{Vigneras}
M. F. Vign\'eras,
\newblock{\em Arithm\'etiques des alg\'ebres de quaternions,}
\newblock{Lec. Note. Math. {\bf 800}, Springer, 1980.}


\bibitem[Vo]{Voight}
J. Voight,
\newblock{\em Shimura Curve Computations,}
\newblock{Arithmetic Geometry (Clay Math. Proc. {\bf 8}), 2009, 103-113.}


\bibitem[Y]{Yang}
Y. Yang,
\newblock{\em Quaternionic loci in Siegel's modular threefolds,}
\newblock{http//www.tims.ntu.edu.tw/download.talk.Summary.pdf, 2015.}


\end{thebibliography}
\end{document}